\documentclass[12pt,a4paper,reqno]{amsart}
\usepackage{latexsym}
\usepackage{amssymb}
\usepackage{enumitem}

\def \zpizq{\left(}
\def \zcizq{\left[}
\def \zpder{\right)}
\def \zcder{\right]}

\def \za{\alpha}
\def \zb{\beta}
\def \zg{\gamma}

\def \zl{\lambda}
\def \zm{\mu}

\def \zr{\rho}

\def \zt{\tau}

\def \zf{\varphi}

\def \zw{\omega}

\def \zL{\Lambda}

\def \zW{\Omega}

\def \zsu{\sum}  
\def \zpr{\prod}

\def \zin{\cap} 
\def \zing{\bigcap} 
\def \zun{\cup}  
  
\def \zex{\wedge}

\def \zdi{\oplus} 
  
\def \zte{\otimes}  

\def \zmm{\pm}
 
\def \zpu{\cdot}  
\def \zpor{\times}
\def \zci{\circ}

\def \zmei{\leq}
\def \zmai{\geq}
\def \zco{\subset}

\def \zcco{\supset}

\def \zpe{\in}

\def \zeq{\equiv}

\def \znoi{\neq}

\def \znoco{\not\subset}

\def \znope{\not\in}

\def \zpar{\partial}
\def \zinf{\infty}

\def \zfl{\rightarrow}

\def \zbv{\mid}

\def \z/{\over}

\def \zdp{\colon}
\def \zps{\dots}

\newtheorem{theorem}{Theorem}
\newtheorem*{theorem*}{Theorem}
\newtheorem{lemma}{Lemma}

\newtheorem{proposition}{Proposition}

\newtheorem{remark}{Remark}
\newtheorem{example}{Example}

\title{Flatness of generic Poisson pairs in odd dimension}

\author{Francisco-Javier~Turiel}
\address[F.J.~Turiel]{
Geometr{\'\i}a y Topolog{\'\i}a,
Facultad de Ciencias,
Campus de Teatinos, s/n,
29071-M{\'a}laga, Spain}
\email{turiel@uma.es}

\thanks{Author is partially supported by
MEC-FEDER grant MTM2013-41768-P, and JA grants FQM-213 and P07-FQM-2863}

\begin{document}

\begin{abstract}
Given a $(m-2)$-form $\zw$ and a volume form $\zW$ on a $m$-manifold one defines
a bi-vector $\zL$ by setting $\zL(\za,\zb)={\frac {\za\zex\zb\zex\zw} {\zW}}$ for any
$1$-forms $\za,\zb$. In this way, locally, a Poisson pair, or  bi-Hamiltonian structure,
$(\zL,\zL_1 )$ is always represented by a couple of $(m-2)$-forms $\zw,\zw_1$ and a 
volume form $\zW$. Here one shows that, for $m\zmai 5$  and odd and $(\zL,\zL_1 )$
generic, $(\zL,\zL_1 )$ is flat if and only if there exists a $1$-form $\zl$ such that
$d\zw=\zl\zex\zw$ and $d\zw_1 =\zl\zex\zw_1$.

Moreover, we use this result for constructing several examples of linear or Lie Poisson
pairs that are generic and non-flat.

 
\end{abstract}

\maketitle

\section{Introduction}\label{sec-1}

Henceforth differentiable means $C^{\zinf}$ in the real case and holomorphic in the 
complex one. Manifolds, real or complex, and objects on them are assumed to be 
differentiable unless another thing is stated.

Poisson pairs or bi-Hamiltonian structures, introduced by F. Magri \cite{MA1} and
I.Gelfand and I. Dorfman \cite{GD1}, are a powerful tool for integrating many
equations from Physics. At the same time the study of their geometric properties,
regardless other aspects, gives rise to several interesting problems, global 
\cite{RI1,RI2} and mostly local, for instance the theory of Veronese webs, notion
introduced by I.Gelfand and I.  Zakharevich in codimension one and later on extended
to any codimension by other authors \cite{GZ1,GZ2,PA1,TU1,TU2,ZA1}.

Often, for applications, it is important to determine {\it in a practical way} whether
generic Poisson pairs are flat or not (see section  \ref{sec-2} for definitions.) 

In even dimension the problem is solved since a local classification is known \cite{TUa},
and flatness is equivalent to say that the Poisson pair defines a $G$-structure. On the
contrary in odd dimension no general and practical criterion of flatness is known, except
for dimension three where a simple explicit obstruction has been constructed by
A. Izosimov \cite{IZ1}. In this work one gives a such criterion by making use of the
differential forms.

More exactly, given a $(m-2)$-form $\zw$ and a volume form $\zW$ on a $m$-manifold 
one defines a bi-vector $\zL$ by setting 
$\zL(\za,\zb)={\frac {\za\zex\zb\zex\zw} {\zW}}$ for any
$1$-forms $\za,\zb$. In this way, locally, a Poisson pair
$(\zL,\zL_1 )$ is always represented by a couple of $(m-2)$-forms $\zw,\zw_1$ and a 
volume form $\zW$. Here one shows that, for $m\zmai 5$  and odd and $(\zL,\zL_1 )$
generic, $(\zL,\zL_1 )$ is flat if and only if there exists a $1$-form $\zl$ such that
$d\zw=\zl\zex\zw$ and $d\zw_1 =\zl\zex\zw_1$.

In a second time we apply this result to generic, linear or Lie, Poisson pairs
(see section \ref{sec-5} for definitions.) It is well known the difficulty for constructing
this kind of Poisson pairs (for a non-trivial example of linear Poisson pair in dimension
five see \cite{KO1}.)
Our result allows to construct, in any dimension $\zmai 5$, examples
of these pairs which are non flat; for instance on the dual space of some Lie 
algebra of truncated polynomial fields in one variable (examples \ref{eje-1} and
\ref{eje-8}.)

Moreover one shows that if the dual space of a non-unimodular Lie algebra of dimension
$\zmai 5$ supports a generic linear Poisson pair, then it supports a non-flat
generic linear Poisson pair too. This result applies to some semi-direct product of the 
affine algebra and an ideal of dimension one (example \ref{eje-7}.)

On the other hand from a Lie algebra one constructs a second one, called
secondary, whose dual space support a generic non-flat linear Poisson pair provided that
some minor conditions hold; just it is the case of the special affine algebra
 (example \ref{eje-6}.) 

In section \ref{sec-8} many examples of generic non-flat Lie Poisson pairs are given, one
of them containing the affine algebra as subalgebra (see remark \ref{rem-6}.)

From our examples follows that for the linear Lie algebra, and therefore for any Lie
algebra ${\mathcal A}$, there exist two other Lie algebras ${\mathcal B}$ and 
${\mathcal B'}$, which contain ${\mathcal A}$ as subalgebra, such that 
${\mathcal B}^*$ support   a generic non-flat linear Poisson pair and  
${\mathcal B'}^*$ a generic non-flat Lie Poisson pair. Thus a natural question arises:
to determine the minimal dimension of such Lie algebras and describe the structural
 relation among ${\mathcal A}$, ${\mathcal B}$ and  ${\mathcal B'}$.

Finally, sections \ref{sec-9} and \ref{sec-10} include a completely description of generic
and non-flat linear Poisson pair or Lie Poisson pair, respectively, in dimension three.

\section{Preliminaries}\label{sec-2}

Let $V$ be a real or complex vector space of dimension $m$ and 
${\mathbb K}={\mathbb R},{\mathbb C}$. Given a volume form $\zW$ on $V$, 
to each $\zw\zpe\zL^{m-2} V^*$ one may associated a bi-vector $\zL$ through
the formula $$\zL(\za,\zb)={\frac {\za\zex\zb\zex\zw} {\zW}}$$
where $\za,\zb\zpe V^*$   and the quotient means the scalar
$a$ such that $\za\zex\zb\zex\zw =a {\zW}$; in this way an isomorphism between 
$\zL^{m-2} V^*$ and $\zL^{2} V$ is defined.
Thus a bi-vector $\zL$ on $V$ can be represented by
a couple of forms $(\zw,\zW)$ where $\zw\zpe\zL^{m-2} V^*$ and  
$\zW\zpe\zL^{m} V^*-\{0\}$.
Of course $(\zw,\zW)$ and $(\zw',\zW')$ represent the same bi-vector if and only if
$\zw'=b\zw$ and $\zW'=b\zW$ for some $b\zpe {\mathbb K}-\{0\}$.

Note that for any $\za\zpe V^*$ its $\zL$-Hamiltonian $\zL(\za,\quad)$ is just the vector
$v_{\za}$ such that $i_{v_{\za}}\zW=-\za\zex\zw$.

\begin{lemma}\label{lem-1}
Consider a bi-vector $\zL$ represented by $(\zw,\zW)$. If 
$\zW=ae_{1}^{*}\zex\zps\zex e_{m}^{*}$
and $\zL=\zsu_{_1\zmei i<j\zmei m}a_{ij}e_i \zex e_j$,
 where $\{e_{1},\zps,e_{m}\}$ is a basis of $V$, then 
$$\zw=\zsu_{_1\zmei i<j\zmei m}(-1)^{i+j-1}aa_{ij}e_{1}^{*}\zex\zps\zex
{\widehat e_{i}^{*}}\zps\zex{\widehat e_{j}^{*}}\zex\zps e_{m}^{*}$$
(as usual terms under hat are deleted.)
\end{lemma}

Recall that a bi-vector can be described by a $r$-form, whose kernel equals the image
of the bi-vector, and a $2$-form , whose restriction to the kernel of the $r$-form is
symplectic; indeed, identify the bi-vector to the dual bi-vector of the restriction of the 
$2$-form.

\begin{lemma}\label{lem-2} On $V$ consider 
$1$-forms $\za_1 ,\zps,\za_ r$ and a $2$-form 
$\zb$ such that $\za_1 \zex\zps\zex\za_ r \zex\zb^k \znoi 0$ where $m=2k+r$ . 
Then $\za_1 \zex\zps\zex\za_ r$, $\zb$ describe a bi-vector that is represented, as well, by 
$(k\za_1 \zex\zps\zex\za_ r \zex\zb^{k-1},\za_1 \zex\zps\zex\za_ r \zex\zb^{k})$. 
\end{lemma}

Consider a couple of bi-vectors $(\zL,\zL_1 )$ on $V$. By definition {\it the rank of}
$(\zL,\zL_1 )$ is the maximum of ranks of $(1-t)\zL+t\zL_1$, $t\zpe\mathbb K$. Note
 that $rank(\zL,\zL_1 )$ equals $rank((1-t)\zL+t\zL_1 )$ for any $t\zpe\mathbb K$
except a finite number of scalars. Therefore, by considering 
$\zL'=(1-a)\zL+a\zL_1 $ and $\zL'_{1}=(1-a_{1})\zL+a_{1}\zL'_{1}$ for suitable 
$a\znoi a_1$, one may assume $rank\zL=rank\zL_1 =rank(\zL,\zL_1 )$ if necessary. 

The classification of couples  $(\zL,\zL_1 )$ is due to Gelfand and Zakharevich  
\cite{GZ2,TU2}; just pointing out that  $(\zL,\zL_1 )$ is the product of 
$m-rank (\zL,\zL_1 )$ Kronecker blocks and, perhaps, a symplectic factor [this last one
only if there exists $b\zpe\mathbb C$ such that 
$rank((1-b)\zL+b\zL_1 )<rank(\zL,\zL_1 )$ or $rank(\zL-\zL_1 )<rank(\zL,\zL_1 )$.]    

Let $M$ be a real or complex manifold of dimension $m$ and $(\zL,\zL_1 )$ a couple of
bi-vectors on it. One will say that {\it at a point} $p$ the couple $(\zL,\zL_1 )$ is:
\begin{enumerate}
\item {\it flat} if, in some coordinates around $p$, $\zL$ and $\zL_1$ can be 
simultaneously written  with constant coefficients.
\item {\it a $G$-structure} if there exists an open neighborhood $A$ of $p$ such that,
for any  $q\zpe A$, the algebraic couples $(\zL,\zL_1 )(q)$ on $T_q M$ and 
$(\zL,\zL_1 )(p)$ on $T_p M$ are isomorphic. 
\item {\it a Poisson pair or a  bi-Hamiltonian structure} if, on some 
open neighborhood of $p$,
$\zL$, $\zL_1$ and $\zL+\zL_1$ are Poisson (structures); in this case $a\zL+b\zL_1$ is 
Poisson for any $a,b\zpe\mathbb K$.     
 \end{enumerate}

When the properties above hold at every point of $M$, $(\zL,\zL_1 )$
is called flat, a $G$-structure or a Poisson pair respectively. 

Note that (1) implies (2) and (3).

One will say that $(\zw,\zw_1 ,\zW)$ {\it represents} $(\zL,\zL_1 )$, when 
$(\zw,\zW)(q)$ and $(\zw_1 ,\zW)(q)$ represent $\zL(q)$ and $\zL_1 (q)$ 
respectively for every $q\zpe M$. 
If $(\zw',\zw'_1 ,\zW')$ is another representative of $(\zL,\zL_1 )$ then 
$(\zw',\zw'_1 ,\zW')=f(\zw,\zw_1 ,\zW)$, that is $\zw'=f\zw$, $\zw'_{1}=f\zw_1$
and $\zW'=f\zW$, where $f$ is a function without zeros. 
The existence of a representative of $(\zL,\zL_1 )$
on $M$ only depends on the existence of a volume form. Since our problem is local, we
may suppose that it is the case without loss of generality. 

\begin{proposition}\label{pro-1}
Let $(\zL,\zL_1 )$ be a Poisson pair on M represented by $(\zw,\zw_1 ,\zW)$.
If $(\zL,\zL_1 )$ is flat at $p$ then, about this point, there is a closed $1$-form $\zl$ 
such that $d\zw=\zl\zex\zw$ and $d\zw_{1}=\zl\zex\zw_{1}$. 

Moreover if a Poisson structure ${\tilde\zL}$ represented by
$({\tilde\zw},{\tilde\zW})$ is flat on a neighborhood of $p$, then around this 
point there exists a function ${\tilde g}$ with no zeros such that ${\tilde g}{\tilde\zw}$
is closed. 
\end{proposition}

\begin{proof} By flatness, and always around $p$, there exists a representative
$(\zw',\zw'_1 ,\zW')$ of $(\zL,\zL_1 )$ such that $d\zw'=d\zw'_{1}=0$. On the other
hand $(\zw,\zw_1 ,\zW)=h(\zw',\zw'_1 ,\zW')$ for some function $h$ with no zeros.
Moreover one can assume $h=\zmm e^{g}$. Then $d\zw=dg\zex\zw$
and $d\zw_1 =dg\zex\zw_1$. 

The second part of proposition \ref{pro-1} is obvious
\end{proof}

The proposition above provides a necessary condition for flatness that in some cases, as
one shows in the next two sections, is sufficient too.

\section{The generic case in odd dimension}\label{sec-3} 

In this section $m=2n-1\zmai 3$, so $n\zmai 2$. Let $(\zL,\zL_1 )$ a couple of Poisson
structures. Suppose $(\zL,\zL_1 )$ {\it generic} (at each point.) This is equivalent to
assume that $(a\zL+b\zL_1 )^{n-1}$ has no zeros for any 
$(a,b){\zpe\mathbb C}^{2}-\{0\}$; so its algebraic model has just a Kronecker block and
no symplectic factor, and $(\zL,\zL_1 )$ defines a $G$-structure. Moreover if a
$1$-form $\zt$ is Casimir for both $\zL$ and $\zL_1$, that is 
$\zL(\zt,\quad)=\zL_{1}(\zt,\quad)=0$, then $\zt=0$. Thus if  $(\zw,\zw_1 ,\zW)$
represents $(\zL,\zL_1 )$ and $\zt\zex\zw=\zt\zex\zw_1 =0$ then $\zt=0$. 

Therefore if there exists $\zl$ such that $d\zw=\zl\zex\zw$ and $d\zw_1 =\zl\zex\zw_1$, 
it is unique. Besides, if $(\zw',\zw'_1 ,\zW')$ represents $(\zL,\zL_1 )$ as well, then we can 
assume $(\zw',\zw'_1 ,\zW')=\zmm e^{g}(\zw,\zw_1 ,\zW)$. So 
$d\zw'=(\zl+dg)\zex\zw'$ and $d\zw'_1 =(\zl+dg)\zex\zw'_1$. This shows that the 
existence of this kind of $1$-form does not depend on the representative while its 
exterior derivative is intrinsic. Consequently $d\zl$ can be constructed even if
representatives are local only.

In general there is no reason for the existence of a such $\zl$ except in dimension three.
More exactly: 

\begin{proposition}\label{pro-2}
Let $(\zL,\zL_1 )$ be a couple of Poisson structures and $(\zw,\zw_1 ,\zW)$ a 
representative. Assume $m=3$ and $(\zL,\zL_1 )$ generic. Then there exists a $1$-form
$\zl$ such that $d\zw=\zl\zex\zw$ and $d\zw_1=\zl\zex\zw_1$ if and only if $(\zL,\zL_1 )$
is a Poisson pair; in this case $Im\zL\zin Im\zL_1 \zco Kerd\zl$.

Moreover $d\zl=0$ if and only if $(\zL,\zL_1 )$ is flat. 
\end{proposition}

\begin{proof} 
In dimension three a bi-vector $\zL'$ represented by $(\zw',\zW')$ is Poisson if and only
if $\zw'$ is completely integrable, that is $\zw'\zex d\zw'=0$. Besides if $\zL'(q)\znoi 0$
then $Im\zL'(q)=Ker\zw'(q)$.

In our case as $\zw\zex d\zw=0$ and the problem is local, there exists a function $\zf$
without zeros such that $d(\zf\zw)=0$. Thus changing of representative allows us to
suppose $\zw$ closed.

Since $Im\zL$ and $Im\zL_1$ are in general position about any point, one may choose 
coordinates $(x_1 ,x_2 ,x_3 )$ such that $\zw=dx_1$ and $\zw_1 =e^{h}dx_2$.
Moreover by modifying the third coordinate if necessary one can suppose 
$\zW=dx_1 \zex dx_2 \zex dx_3$. The couple $(\zL,\zL_1 )$ is a Poisson pair if and only if
$\zL+\zL_1 $ is Poisson, that is if and only if $(\zw+\zw_1 )\zex d(\zw+\zw_1 )=0$
(note that $(\zw+\zw_1, \zW)$ represents $\zL+\zL_1 $,) which is equivalent to say that
$h=h(x_1, x_2 )$.

On the other hand as $d\zw=0$ the $1$-form $\zl$, if any, has to be equal $gdx_1$, which
implies that $d\zw_1=e^{h}dh\zex dx_2 =ge^{h}dx_1 \zex dx_2$. In other words a such
$\zl$ exists just when $h=h(x_1, x_2 )$; in this case $\zl=(\zpar h/\zpar x_1 )dx_1$. That
proves the first part of proposition \ref{pro-2}. 

Now suppose $d\zl=0$; then $\zpar^{2}h/\zpar x_2 \zpar x_1 =0$ and 
$h=h_1 (x_1 )+h_2 (x_2 )$ so $e^{-h_1}\zw$ and $e^{-h_1}\zw_1$ are closed.
Therefore changing of representative allows to suppose $\zw$, $\zw_1$ closed. 
Consequently there exist coordinates $(y_1 ,y_2 ,y_3 )$ such that $\zw=dy_1$ and
$\zw_1 =dy_2$. By modifying the third coordinate if necessary we may suppose
$\zW=dy_1 \zex dy_2 \zex dy_3$; now is obvious that $(\zL,\zL_1 )$ is flat. The converse 
follows from proposition \ref{pro-1}.  
\end{proof}

\begin{remark}\label{rem-1}
{\rm By proposition \ref{pro-2} in dimension three $d\zl$ is an invariant of the Poisson
pair $(\zL,\zL_1 )$, which vanishes just when $(\zL,\zL_1 )$ is flat. 
A straightforward computation shows that, up to multiplicative constant, this invariant equals
the curvature form introduced by Izosimov \cite{IZ1}}
\end{remark}

For odd dimension greater than or equal to five one has: 

\begin{theorem}\label{teo-1}
Consider a Poisson pair $(\zL,\zL_1 )$ on a manifold $M$ of dimension
$m=2n-1\zmai 5$ represented by $(\zw,\zw_1 ,\zW)$. Assume that $(\zL,\zL_1 )$ is 
generic everywhere. Then $(\zL,\zL_1 )$ is flat if and only if there exists a $1$-form $\zl$ 
such that $d\zw=\zl\zex\zw$ and $d\zw_1 =\zl\zex\zw_1$.  
\end{theorem}

\begin{proof} 
The existence of $\zl$ when $(\zL,\zL_1 )$ is flat follows from proposition \ref{pro-1}. 
Conversely suppose that $\zl$ exists; for proving the flatness, which is a local question,
it will be enough to show that the Veronese web associated is flat \cite{GZ1,GZ2,
TU1,TU2}.

One starts proving that $d\zl=0$. For every $a\zpe \mathbb K$ the Poisson structure
$\zL+a\zL_1$ is represented by $(\zw+a\zw_1 ,\zW)$ so, by proposition \ref{pro-1},
(locally) there is a function $g_a$ with no zeros such that $g_a (\zw+a\zw_1 )$ is
closed. Therefore $0=(\zl+(dg_a )/g_a )\zex g_a (\zw+a\zw_1 )$ and  
$\zl+(dg_a )/g_a $ has to be a Casimir of $\zL+a\zL_1$.

This implies that $d\zl=d(\zl+(dg_a )/g_a )$ is divisible by any ($1$-form) Casimir of
$\zL+a\zL_1$. Thus $d\zl$ is divisible by every Casimir of
$\zL+a\zL_1$, $a\zpe \mathbb K$. But at each point all these Casimirs span a vector
space of dimension$\zmai 3$. Since $d\zl$ is a $2$-form necessarily vanishes.

As it is known, given any different and non-vanishing real numbers $a_1 ,\zps,a_n$,
around each point $p\zpe M$, there exist coordinates 
$(x,y)=(x_1 ,\zps,x_n ,y_1 ,\zps,y_{n-1})$ and functions $f_1 ,\zps,f_n$ only depending
on $x$ such that $\zL$ is given by 
$(a_1 \zpu\zpu\zpu a_n)\zsu_{j=1}^{n}a_{j}^{-1}f_j dx_j$ and
$\zsu_{j=1}^{n-1}dx_j \zex dy_j$ while  $\zL_1$ is given by 
$\za=\zsu_{j=1}^{n}f_j dx_j$ and
$\zsu_{j=1}^{n-1}a_j dx_j \zex dy_j$, where $f_1 ,\zps,f_n$ have no zeros, $d\za=0$
and $\za\zex d(\za\zci J)=0$ when $\za$ is regarded as a $1$-form in variables
$x=(x_1 ,\zps,x_n )$ and  $J=\zsu_{j=1}^{n}a_{j}(\zpar/\zpar x_{j})\zte dx_j$; that
is $\za\zci J=\zsu_{j=1}^{n}a_{j}f_j dx_j$ and 
$\za\zci J^{-1}=\zsu_{j=1}^{n}a_{j}^{-1}f_j dx_j$ 
(see page 893 of \cite{TU1} and section 3 of \cite{TU2}). For simplifying
computations we choose $a_1 ,\zps,a_n$ in such a way that $a_1 \zpu\zpu\zpu a_n =1$. 

By lemma \ref{lem-2} the Poisson structure $\zL$ is represented by 
$$\zw'=(n-1)(\za\zci J^{-1})\zex(\zsu_{j=1}^{n-1}dx_j \zex dy_j )^{n-2}$$
$$\zW'=(\za\zci J^{-1})\zex(\zsu_{j=1}^{n-1}dx_j \zex dy_j )^{n-1}$$
and $\zL_1$ by
$$\zw'_1 =(n-1)\za\zex(\zsu_{j=1}^{n-1}a_j dx_j \zex dy_j )^{n-2}$$
$$\zW'_1 =\za\zex(\zsu_{j=1}^{n-1}a_j dx_j \zex dy_j )^{n-1}\, .$$

Note that $\zW'=\zW'_1$ because $a_1 \zpu\zpu\zpu a_n =1$. On the other hand 
$d\zw'_1 =0$; therefore for the representative $(\zw',\zw'_1 ,\zW')$ there exists a 
$1$-form $\zl'=\zf\za$ such that $d\zw'=\zl'\zex\zw'$. But $d\zl'=d\zl=0$ so 
$\zl'=dg$ for some function $g=g(x)$ such that $\za\zex dg=0$, 
and $e^{-g}\zw'$, $e^{-g}\zw'_1$ will be closed. 

Observe that ${\tilde\za}=e^{-g}\za$ is closed, $e^{-g}f_1 ,\zps,e^{-g}f_n$ have no 
zeros, ${\tilde\za}\zex d({\tilde\za}\zci J)=0$ while $\zL$, $\zL_1$ are given by 
${\tilde\za}\zci J^{-1}$, $\zsu_{j=1}^{n-1}dx_j \zex dy_j$ and 
${\tilde\za}$, $\zsu_{j=1}^{n-1}a_j dx_j \zex dy_j$ respectively. In other words, 
considering ${\tilde\za}$ instead $\za$ and calling it $\za$ again allows to assume, without
loss of generality, $\zl'=0$ and $\zw'$ closed.

As $Ker(\za\zci J^{-1})=Im\zL$ is involutive $d(\za\zci J^{-1})=\zt\zex(\za\zci J^{-1})$
for some $1$-form $\zt$. Hence $0=d\zw'=\zt\zex\zw'$ since
$(\zsu_{j=1}^{n-1}dx_j \zex dy_j )^{n-2}$ is closed. Therefore $\zt$ is a Casimir of 
$\zL$ and $\zt=h\za\zci J^{-1}$ for some function $h$. But in this case 
$\zt\zex(\za\zci J^{-1})=0$ and $\za\zci J^{-1}$ will be closed. In other words
$$0=d(\za\zci J^{-1})= \zsu_{1\zmei k<j\zmei n}\zcizq 
a_{j}^{-1}(\zpar f_{j}/\zpar x_{k})
-a_{k}^{-1}(\zpar f_{k}/\zpar x_{j})\zcder dx_k \zex dx_j              $$
so $a_{j}^{-1}(\zpar f_{j}/\zpar x_{k})-a_{k}^{-1}(\zpar f_{k}/\zpar x_{j})=0$.
Since $(\zpar f_{j}/\zpar x_{k})=(\zpar f_{k}/\zpar x_{j})$, necessarily
$(\zpar f_{j}/\zpar x_{k})=0$ if $j\znoi k$.
Thus $f_j =f_j (x_j )$, $j=1,\zps,n$. 

Recall that the web associated to $(\zL,\zL_1 )$ is completely determined by $J$ and
$\za$ regarded on the (local) quotient manifold $N$ of $M$ by the foliation
$\zing_{t\zpe\mathbb K}Im(\zL+t\zL_1 )$ (see \cite{TU1,TU2} again;) of course
$(x_1 ,\zps,x_n )$ can be regarded too as coordinates on $N$ around  $q$, where $q$ is
the projection on $N$ of point $p$.
Now consider functions ${\tilde x}_1 ,\zps,{\tilde x}_n$ such that 
$d{\tilde x}_j =f_j (x_j )dx_j$, $j=1,\zps,n$; then $({\tilde x}_1 ,\zps,{\tilde x}_n )$ are 
coordinates on $N$ about $q$, $\za=d{\tilde x}_1 +\zps+d{\tilde x}_n$ and 
$J=\zsu_{j=1}^{n}a_j (\zpar/\zpar {\tilde x}_{j})\zte d{\tilde x}_j$, which shows
the flatness of the web associated to $(\zL,\zL_1 )$. 
\end{proof}

\begin{example}\label{eje-A}
{\rm On ${\mathbb K}^5$ with coordinates $x=(x_1 ,x_2 ,x_3 ,x_4 ,x_5 )$ consider 
the  bi-vectors
$$\zL=\zpizq x_2 {\frac {\zpar} {\zpar x_1}}+x_3 {\frac {\zpar} {\zpar x_3 }}\zpder
\zex{\frac {\zpar} {\zpar x_4}}
+\zpizq x_1 {\frac {\zpar} {\zpar x_1}}+x_2 {\frac {\zpar} {\zpar x_2}}
+x_3 {\frac {\zpar} {\zpar x_3 }}\zpder\zex{\frac {\zpar} {\zpar x_5}}$$

$$\zL_1 =x_1 {\frac {\zpar} {\zpar x_2 }}\zex{\frac {\zpar} {\zpar x_4}}
+\zpizq x_1 {\frac {\zpar} {\zpar x_1}}+x_2 {\frac {\zpar} {\zpar x_2}}
\zpder\zex{\frac {\zpar} {\zpar x_5}}\, .$$
\vskip .3truecm

First note that $\zL+t\zL_1 =X_t \zex (\zpar /\zpar x_ 4)+Y_t \zex (\zpar /\zpar x_ 5)$
where $X_t$, $Y_t$, $(\zpar /\zpar x_ 4)$ and $(\zpar /\zpar x_ 5)$ commute among 
them, which implies that every $\zL+t\zL_1$ is a Poisson structure and $(\zL,\zL_1 )$ a 
Poisson pair (in fact a Lie Poisson pair, see section \ref{sec-5}.)

Let $A$ be the complement of the union of hyperplanes $x_k =0$, $k=1,2,3$, and 
in addition to these  $x_2 =(a_j +1)x_1$, $j=1,2$, where $a_1 ,a_2$ are the roots 
of $t^2 +t+1=0$, if ${\mathbb K}={\mathbb C}$. Then $(\zL,\zL_1 )$ is generic on $A$. 

Set $\zW=dx_1 \zex dx_2 \zex dx_3 \zex dx_4 \zex dx_5$. Then $\zL$ is represented by
$$\zw=(-x_3 dx_1 \zex dx_2 +x_2 dx_1 \zex dx_3 -x_1 dx_2 \zex dx_3)\zex dx_4
\hskip 2truecm$$
$$\hskip 2truecm+(x_3 dx_1 \zex dx_2 +x_2 dx_2 \zex dx_3 )\zex dx_5$$ and $\zL_1$ by  
$\zw_1 =(x_2 dx_1 -x_1 dx_2 )\zex dx_3 \zex dx_4 -x_1 dx_1 \zex dx_3 \zex dx_5\, .$

Moreover
$$d\zw=-3dx_1 \zex dx_2 \zex dx_3 \zex dx_4+dx_1 \zex dx_2 \zex dx_3 \zex dx_5$$ and 
$$d\zw_1 =-2dx_1 \zex dx_2 \zex dx_3 \zex dx_4 
=(2x_{1}^{-1}dx_1 )\zex\zw_1\, .$$

As $dx_3$ is a Casimir of $\zL_1$, if $d\zw_1 =\zl\zex\zw_1$ then
$\zl=2x_{1}^{-1}dx_1 +fdx_3$.

On the other hand if we assume $d\zw=\zl\zex\zw$ then 
$\zL(\zl,\quad)=(\zpar /\zpar x_ 4)+3(\zpar /\zpar x_ 5)$ since the contraction of 
$-\zL(\zl,\quad)$ and $\zW$ equals $\zl\zex\zw=d\zw$. But at the same time
$$\zL(\zl,\quad)={\frac {2x_2} {x_1}}{\frac {\zpar} {\zpar x_4 }}
+2{\frac {\zpar} {\zpar x_5}}+x_3 f\zpizq {\frac {\zpar} {\zpar x_4 }}+
{\frac {\zpar} {\zpar x_5}}\zpder$$
which implies that $(2x_{1}^{-1}x_2 +x_3 f)$ and $x_3 f$ have to be constant, 
{\it contradiction}. In short $(\zL,\zL_1 )$ is not flat at any point of $A$.} 
\end{example}

A practical criterion for local flatness is the following:

\begin{lemma}\label{lem-3}
Let $M$, $(\zL,\zL_1 )$ and $(\zw,\zw_1 ,\zW)$ be as in theorem \ref{teo-1}. On some 
neighborhood of a point $p$ of $M$ consider a vector field $X$ tangent to $Kerd\zw$ and 
a $1$-form $\za$ Casimir of $\zL_1$, both of them non-vanishing at $p$. Assume 
$d\zw(p)\znoi 0$ and $d\zw_1 =0$ on this neighborhood. Then $(\zL,\zL_1 )$ is flat at $p$
if and only if $\zL(\za,\quad)$, about $p$, is functionally proportional to $X$. 
\end{lemma}

\begin{proof}
In this proof objects will be considered on a suitable neighborhood of $p$. First observe that 
$\zL(\za,\quad)$ does not vanishes anywhere because no common Casimir of $\zL$ and
$\zL_1$ other that zero exists. Recall that  $i_{\zL(\za,\quad)}\zW=-\za\zex\zw$.
Thus there is a function $f$ such that $d\zw=f\za\zex\zw$ if and only if $\zL(\za,\quad)$
and $X$ are functionally proportional. When $\zL(\za,\quad)$ and $X$ are proportional 
it suffices setting $\zl=f\za$.

Conversely if $\zl$ exists necessarily $\zl=f\za$ since $d\zw_1 =0$, so
$d\zw=\zl\zex\zw=f\za\zex\zw$ and $\zL(\za,\quad)$ and $X$ are proportional. 
\end{proof}

\section{Other cases}\label{sec-4}
The foregoing results extend to any analytic Poisson pair $(\zL,\zL_1 )$ on a
$m$-manifold $M$ provided that $rank(\zL,\zL_1 )=m,m-1$. Indeed, if $m$ is even and
$rank(\zL,\zL_1 )=m$, then $(\zL,\zL_1 )$ is flat at a point $p\zpe M$ if and only if it
defines a $G$-structure. This is a straightforward consequence of the classification of
pairs of compatible symplectic forms \cite{TUa} since, up to linear combination, one
may suppose symplectic $\zL$ and $\zL_1$; in this case analyticity is not need. 
On the other hand:

\begin{theorem}\label{teo-2}
On a real analytic or complex manifold $M$, of odd dimension $m$, consider a point 
$\zpe M$ and a Poisson pair $(\zL,\zL_1 )$, represented by $(\zw,\zw_1 ,\zW)$,
whose rank equals $m-1$. Then $(\zL,\zL_1 )$is flat at $p$ if and only if
$(\zL,\zL_1 )$ defines a $G$-structure at $p$ and there is a closed $1$-form $\zl$, 
about $p$, such that $d\zw=\zl\zex\zw$ and $d\zw_1 =\zl\zex\zw_1$.
\end{theorem}

\begin{proof}
As the problem is local, up to linear combination, one may suppose
$rank\zL=rank\zL_1 =rank(\zL,\zL_1 )=m-1$. Obviously the conditions of the theorem
are necessary (see proposition \ref{pro-1}.) Conversely, since $(\zL,\zL_1 )$ defines a 
$G$-structure at $p$, from the product theorem for Poisson pairs \cite{TU3}
follows that, always about $p$, $(M,\zL,\zL_1 )$ splits into a product of two
Poisson pairs $(M',\zL',\zL'_1 )\zpor (M'',\zL'',\zL''_1 )$, $p=(p',p'')$, the 
first one Kronecker with a single block, so generic, and symplectic the second one.

Moreover $(\zL'',\zL''_1 )$ defines a $G$-structure at $p''$ because $(\zL,\zL_1 )$ does
at $p$; therefore it is flat at $p''$ and one can choose a representative
$(\zw'',\zw''_1 ,\zW'')$ with $d\zw''=d\zw''_1 =0$.

If $dimM'=1$ the proof is finished since $\zL'=\zL'_1 =0$; therefore assume
$dimM'\zmai 3$. Let $(\zw',\zw'_1 ,\zW')$ be a representative of $(\zL',\zL'_1 )$. Then 
with the obvious identification (by means of the pull-back by the canonical projections forms
on $M'$ or $M''$ can be regarded as forms on $M$)
$$(\zw'\zex\zW''+\zW'\zex\zw'',\zw'_1 \zex\zW''+\zW'\zex\zw''_1 ,\zW'\zex\zW'')$$
represents $(\zL,\zL_1 )$. 

Note that the existence of a closed form as in theorem \ref{teo-2} happens for any
representative because all of them are functionally proportional. Let us denote by
${\tilde\zl}$ that corresponding to the representative above and $\zl'$ its restriction to
$M'\zpor \{p''\}$, identified to $M'$ in the obvious way. Then $\zl'$ is closed,
$d\zw'=\zl'\zex\zw'$ and $d\zw'_1 =\zl'\zex\zw'_1$. Therefore 
by proposition \ref{pro-2} and theorem \ref{teo-1}
$(\zL',\zL'_1 )$ is flat at $p'$, which allows to conclude the flatness 
of $(\zL,\zL_1 )$ at $p$.
\end{proof} 

\begin{remark}\label{rem-2}
{\rm In the real case the analyticity is needed for applying the product theorem; 
nevertheless an unpublished results by the author states that if $(\zL,\zL_1 )$ defines a
$G$-structure and $rank(\zL,\zL_1 )=dimM-1$ then the product theorem holds in the
$C^{\zinf}$-category too.

If $dimM'=3$ then there always exists $\zl$ because this fact essentially depends on
$(M',\zL',\zL'_1 )$. By the same reason when $dimM'\zmai 5$ the condition 
$d\zl=0$ is unnecessary.}
\end{remark}

\section{Pairs on the dual space of a Lie algebra}\label{sec-5}

The remainder of this work is essentially devoted to the generic case in odd dimension when
 $\zL$ is linear and $\zL_1$ linear or constant.

\begin{lemma}\label{lem-4}
Consider a couple of analytic bi-vectors $(\zL,\zL_1 )$ on a connected non-empty open
set $A\zco{\mathbb K}^m$, $m=2n-1\zmai 3$. If $(\zL,\zL_1 )(p)$ is generic for some
$p\zpe A$, then the set of points $q\zpe A$ such that $(\zL,\zL_1 )(q)$ is generic is
open and dense. 
\end{lemma}

\begin{proof} 
Fixed $s\zpe \mathbb K$ the equation  $(\zL+s\zL_1 )(\zl,\quad)=0$ is an 
homogeneous linear system, on $({\mathbb K}^m )^*$,  with analytic
coefficients (like functions on $A$.) Since $(\zL+s\zL_1 )^{n-1}(p)\znoi 0$, at each 
point near $p$ the vector subspace of its solutions has dimension one. Thus around $p$ 
there is a solution $\zl_{s}=\zsu_{j=1}^{m}f_j dx_j$, with  $\zl_{s}(p)\znoi 0$,  where
every $f_j$ is a rational function of the coefficients functions of $\zL$ and $\zL_1$.
Multiplying by a suitable polynomial allows to assume, without loss of generality, that
$f_1 ,\zps,f_n$ are polynomials in the coefficients functions of $\zL$ and $\zL_1$.
Therefore $\zl_s$ is defined on $A$ and by analyticity
$(\zL+s\zL_1 )(\zl_s ,\quad)=0$ everywhere.

Consider $n$ different scalars $s_1 ,\zps,s_n$. Then 
$(\zl_{s_1} \zex\zps\zex\zl_{s_n} )(p)\znoi 0$ and by analyticity
$\zl_{s_1} \zex\zps\zex\zl_{s_n}$ does not vanish at any point 
of a dense open set $A'$.

By a similar reason each $(\zL+s_j \zL_1 )^{n-1}$ does not vanish at any point of a
dense open set $A_j$. Take any $q\zpe A'\zin A_1 \zin\zps\zin A_n$; then 
$(\zL+s_j \zL_1 )^{n-1}(q)\znoi 0$, $j=1,\zps,n$, and 
$(\zl_{s_1} \zex\zps\zex\zl_{s_n} )(q)\znoi 0$, which only is 
possible if $(\zL,\zL_1 )(q)$ is generic.
\end{proof}

\begin{remark}\label{rem-3}
{\rm It is easily seen that lemma \ref{lem-4} holds in analytic connected manifolds too.}
\end{remark}

Let ${\mathcal A}$ be a Lie algebra of finite dimension and $G$ a connected Lie group of
algebra ${\mathcal A}$. Then any element of $\zL^{r}{\mathcal A}$, respectively
$\zL^{r}{\mathcal A}^*$, can be regarded like a left invariant $r$-field, 
respectively $r$-form, on $G$. Thus we may consider the Lie and exterior derivatives on  
${\mathcal A}$ and from the formulas on $G$  deduce the corresponding formulas on  
${\mathcal A}$. Here we adopt the differentiable point of view or the algebraic one 
depending on the convenience for working with. Recall that closed means cocycle
and exact cobord. One say that $\zr\zpe {\mathcal A}^*$ is a {\it contact form} if
$dim {\mathcal A}=2k+1$ and $\zr\zex(d\zr)^k$ is a volume form, while 
$\zb\zpe\zL^{2} {\mathcal A}^*$ is called {\it symplectic} when $d\zb=0$ 
and $rank\zb=dim {\mathcal A}$. 

Remember that on ${\mathcal A}^*$ one defines the {\it Lie-Poisson structure} $\zL$ by 
considering  the  elements of ${\mathcal A}$ as linear functions and setting 
$\zL(a,b)=[a,b]$. In the  coordinates associated to any basis of ${\mathcal A}^*$ the 
coefficients of $\zL$ are linear functions. Conversely any Poisson structure on a vector 
space  with linear coefficients is obtained in this way.

On the other hand a $2$-form $\zb\zpe\zL^{2} {\mathcal A}^*$ can be seen as a
Poisson structure $\zL$ on  ${\mathcal A}^*$ with constant coefficients. Remember that
$(\zL,\zL_1 )$ is compatible if and only if $d\zb=0$. In this case $(\zL,\zL_1 )$ will be
named {\it a linear Poisson pair or a linear  bi-Hamiltonian structure}.

Other option consists in considering a second Lie algebra structure $[\quad,\quad]_1$ on
${\mathcal A}$ and its associated Lie-Poisson structure $\zL_1$ on ${\mathcal A}^*$.
Then $(\zL,\zL_1 )$ is compatible if and only if  $[\quad,\quad]+[\quad,\quad]_1$ 
defines a Lie algebra structure. In this case $(\zL,\zL_1 )$ will be called {\it a Lie
Poisson pair or a Lie  bi-Hamiltonian structure}.

In both cases, because lemma \ref{lem-4}, when $dim{\mathcal A}=2n-1\zmai 3$ one
will say that  $(\zL,\zL_1 )$ is {\it generic} if it is generic at some point (so almost 
everywhere), and {\it flat} if it is flat at every generic point.

Set $I_0 =\{a\zpe{\mathcal A}\zbv tr[a,\quad]=0\} $ where $tr[a,\quad]$ is the trace of
the adjoint endomorphism $[a,\quad]\zdp b\zpe{\mathcal A}\zfl [a,b]\zpe{\mathcal A}$.
Note that $I_0$ is an ideal, which contains the derived ideal of ${\mathcal A}$, and 
that we will call the {\it unimodular ideal of ${\mathcal A}$}. 
Its codimension equals zero when ${\mathcal A}$ is unimodular and one otherwise.
As Lie algebra $I_0$ is unimodular.  

\begin{lemma}\label{lem-5}
Let $\{e_1 ,\zps,e_m \}$ be a basis of a Lie algebra ${\mathcal A}$ and 
$(x_1 ,\zps,x_m )$ the coordinates on ${\mathcal A}^*$ associated to the dual basis. 
Set $\zW=dx_1 \zex\zps\zex dx_m$ and 
$X=\zsu_{j=1}^{m}(tr[e_j ,\quad ])(\zpar/\zpar x_j )$. Assume that $(\zw,\zW)$
represents the Lie-Poisson structure $\zL$ of  ${\mathcal A}^*$. Then
$d\zw=i_X \zW$.

Therefore $d\zw=0$ if and only if  ${\mathcal A}$ is unimodular. 
\end{lemma}

\begin{proof}      
Suppose that $[e_i ,e_j ]=\zsu_{k=1}^{m}a_{ij}^{k}e_k$; then 
$$\zL=\zsu_{1\zmei i<j\zmei m}\zpizq\zsu_{k=1}^{m}a_{ij}^{k}x_k \zpder
{\frac {\zpar} {\zpar x_i}}\zex {\frac {\zpar} {\zpar x_j}}$$
and by lemma \ref{lem-1}
$$\zw=\zsu_{1\zmei i<j\zmei m}(-1)^{i+j-1}
\zpizq\zsu_{k=1}^{m}a_{ij}^{k}x_k \zpder
dx_1 \zex\zps\zex{\widehat dx_{i}}\zex\zps\zex{\widehat dx_{j}}\zex\zps\zex dx_m .$$

Finally a straightforward computation yields
$$d\zw=\zsu_{j=1}^{m}(-1)^{j}\zpizq\zsu_{i=1}^{m}a_{ij}^{i}\zpder
dx_1 \zex\zps\zex{\widehat dx_{j}}\zex\zps\zex dx_m$$
$$=\zsu_{j=1}^{m}(-1)^{j-1}(tr[e_j ,\quad])
dx_1 \zex\zps\zex{\widehat dx_{j}}\zex\zps\zex dx_m.$$
\end{proof}

\begin{remark}\label{rem-4}
{\rm Assume $m\zmai 3$ and odd. First consider a constant Poisson structure $\zL_1$ on 
${\mathcal A^*}$ such that $(\zL, \zL_1 )$ is compatible and generic. If ${\mathcal A}$
is unimodular, as $\zL_1$ is represented by a constant $(m-2)$-form $\zw_1$, by
proposition \ref{pro-2} and theorem \ref{teo-1} the pair $(\zL, \zL_1 )$ is flat.

Now suppose $\zL_1$ given by a second Lie algebra structure $[\quad,\quad]_1$ and
$(\zL, \zL_1 )$ compatible and generic. Again by proposition \ref{pro-2} and theorem 
\ref{teo-1}, when  $[\quad,\quad]$ and  $[\quad,\quad]_1$ are unimodular,  
$(\zL, \zL_1 )$ is flat. Flatness happens as well if the derived ideal of $[\quad,\quad]$
equals ${\mathcal A}$; indeed, for $a\znoi 0$ small enough the derived ideal of
$[\quad,\quad]+a[\quad,\quad]_1$ equals ${\mathcal A}$ too; therefore $[\quad,\quad]$ 
and $[\quad,\quad]+a[\quad,\quad]_1$ are unimodular and$(\zL, \zL+a\zL_1 )$ flat.}
\end{remark}

A couple of $2$-forms $\zb,\zb_1$ on a vector space $V$, of dimension $2n-1\zmai 3$,
is named {\it generic} if it is generic as bi-vectors on $V^*$, that is if
$(s\zb+t\zb_1 )^{n-1}\znoi 0$ for any $(s,t)\zpe {\mathbb C}^2 -\{0\}$; note that
when $\zb_{1}^{n-1}\znoi 0$ it suffices to check that
$(\zb+t\zb_1 )^{n-1}\znoi 0$ for any $t\zpe {\mathbb C}$.

Observe that if $\zL$ is the Lie-Poisson structure on the dual space ${\mathcal B}^*$
of a Lie algebra ${\mathcal B}$, then $\zL(\za)=-d\za$ for each $\za\zpe {\mathcal B}^*$,
where $d\za$ is regarded as constant bi-vector on ${\mathcal B}^*$. Therefore, when
$dim{\mathcal B}=2n-1\zmai 3$, a linear (respectively Lie) Poisson pair
$(\zL,\zL_1 )$ on ${\mathcal B}^*$, where $\zL_1$ is associated to 
$\zb\zpe\zL^2 {\mathcal B}^*$ (respectively to a second bracket $[\quad,\quad]_1$),
is generic at $\za\zpe {\mathcal B}^*$ if and only if $(d\za,\zb)$ (respectively
$(d\za,d_1 \za)$) is generic.

\begin{example}[{\it Truncated algebra}]\label{eje-1}
{\rm Let ${\mathcal P}$ be the Lie algebra of vector fields $f\zpu(\zpar/\zpar u)$, on 
$\mathbb K$, such that $f$ is a polynomial in $u$ and $f(0)=0$. Set 
${\mathcal P}_m =u^{m}{\mathcal P}$, $m\zpe \mathbb N$, which is an ideal of
${\mathcal P}$ of codimension $m$. The quotient 
${\mathcal A}={\mathcal P}/{\mathcal P}_m$ (we omit the subindex for sake of 
simplicity) is a Lie algebra of dimension $m$ that we called the truncated Lie algebra (of
  ${\mathcal P}$ at order $m$). Denote by $e_k$ the class of $u^k (\zpar/\zpar u)$; then
$\{e_1 ,\zps,e_m \}$ is a basis of ${\mathcal A}$ and $[e_i ,e_j ]=(j-i)e_{i+j-1}$
if $i+j\zmei m+1$ and zero otherwise. 

Throughout this example one will assume $m=2n-1\zmai 5$. Then 
$de_{m}^* =-\zsu_{j=1}^{n-1}2(n-j)e_j^* \zex e_{2n-j}^*$ and
$de_{m-1}^* =-\zsu_{j=1}^{n-1}(2(n-j)-1)e_j^* \zex e_{2n-j-1}^*$. Moreover
$(de_m^* ,de_{m-1}^* )$ is generic; this follows from the next lemma since 
$(de_m^* )^{n-1}$ and $(de_{m-1}^* )^{n-1}$ do not vanish. 

\begin{lemma}\label{lem-6}
For every $t\zpe{\mathbb C}-\{0\}$ one has 
$(e_m^* +te_{m-1}^* )\zex(de_m^* +tde_{m-1}^* )^{n-1}\znoi 0$.
\end{lemma}

\begin{proof}
Regard ${\mathcal A}$ like a vector subspace of the $2n$-dimensional vector space $V$ 
of basis $\{e_1 ,\zps,e_m ,{\tilde e}\}$  and ${\mathcal A}^*$ as a vector subspace of
$V^*$ in the obvious way ($\za({\tilde e})=0$ for every $\za\zpe{\mathcal A}^*$). Set
$\zt= -\zsu_{j=1}^{n-1}2(n-j)e_j^* \zex e_{2n-j}^*+e_{n}^* \zex{\tilde e}^*$, which 
is a symplectic form, and $\zt_1 =-\zsu_{j=1}^{n-1}(2(n-j)-1)e_j^* \zex e_{2n-j-1}^*$.
Let $J$ be the endomorphism given by the formula $\zt_1 (v,w)=\zt(Jv,w)$. Then
$J=\zsu_{j=1}^{n-2}a_j e_{j+1}\zte e_{j}^* +a{\tilde e}\zte e_{n-1}^*
+\zsu_{k=n}^{2(n-1)}b_k e_{k+1}\zte e_{k}^*$ for some non-vanishing scalars
$a_j$,  $a$, $b_k$, $j=1,\zps,n-2$, $k=n,\zps,2(n-1)$. Note that $J$ is nilpotent so every
$I+tJ$, $t\zpe \mathbb C$, is invertible and each $\zt+t\zt_1$, $t\zpe \mathbb C$,
symplectic. 

Moreover the vector subspace spanned by $\{e_1 ,\zps,e_{n-1}, {\tilde e}\}$
is Lagrangian for all symplectic form $\zt+t\zt_1$, $t\zpe \mathbb C$. Thus there exists a  
vector $v(t)=\zsu_{j=1}^{n-1}c_j (t)e_j +c(t) {\tilde e}$ such that
$i_{v(t)}(\zt+t\zt_1 )=e_{m}^* +te_{m-1}^*$, $t\zpe \mathbb C$. It is easily checked
that $c(t)\znoi 0$ when $t\znoi 0$; in this case $v(t)\znope{\mathcal A}$ and the 
restriction of $i_{v(t)}(\zt+t\zt_1 )^n$ to ${\mathcal A}$ does not vanish. But this 
restriction equals $n(e_m^* +te_{m-1}^* )\zex(de_m^* +tde_{m-1}^* )^{n-1}$.  
\end{proof}

On ${\mathcal A}^*$ consider coordinates $(x_1 ,\zps,x_m )$ associates to the basis
$\{e_{1}^* ,\zps,e_{m}^* \}$ and the Lie -Poisson structure $\zL$. Set 
$\zL_1 \zeq de_{m}^*$. As $(de_{m}^* ,de_{m-1}^* )$ is generic, $(\zL,\zL_1 )$ is
generic at $e_{m-1}^* \zpe{\mathcal A}^*$. Besides {\it it is not flat at this point}.
Indeed, $dx_n$ is a Casimir of $\zL_1$ and 
$$\zL(dx_n ,\quad)= (1-n)x_n {\frac {\zpar} {\zpar x_1 }}
+ (2-n)x_{n+1}{\frac {\zpar} {\zpar x_2 }}
+\zsu_{j=3}^{n-1}f_j  {\frac {\zpar} {\zpar x_j }}\, .$$

In our case the vector field $X$ of lemma \ref{lem-5} equals $a(\zpar/\zpar x_1 )$ with
$a\znoi 0$. Since $\zL(dx_n ,\quad)$ and $\zpar/\zpar x_1$ are not proportional about
$e_{m-1}^*$, the non-flatness of $(\zL,\zL_1 )$ at $e_{m-1}^*$ follows from 
lemma \ref{lem-3} (see remark below).} 
\end{example}

\begin{remark}\label{rem-5}
{\rm Recall that a linear vector field on a vector space $V$ is proportional, about some 
point $p\zpe V$, to a constant vector field if and only if  the associated endomorphism has 
rank zero or one. Let  ${\mathcal B}$ be a finite dimensional Lie algebra and $\zL'$ the
Lie-Poisson structure of  ${\mathcal B}^*$. Consider a constant $1$-form $\za$
on  ${\mathcal B}^*$; then there exists just a $b_{\za}\zpe V$ such that $\za$ is the
exterior derivative of $b_{\za}$ regarded as a linear function on ${\mathcal B}^*$. On 
the other hand, $\zL'(\za,\quad)$ is a linear vector field whose associated endomorphism
is the dual map of $[b_{\za},\quad]\zdp{\mathcal B}\zfl{\mathcal B}$. Thus 
$\zL'(\za,\quad)$ is proportional, about some $p\zpe{\mathcal B}^*$, to a constant 
vector field if and only if $rank[b_{\za},\quad]\zmei 1$.

Therefore if $\zb\zpe\zL^{2}{\mathcal B}^*$ is a cocycle and $\za$ a Casimir of
$\zL''\zeq\zb$, that is $\zb(b_{\za},\quad)=0$, such that $rank[b_{\za},\quad]\zmai 2$,
then the linear vector field $\zL'(\za,\quad)$ is never proportional to a constant vector
field about any element of ${\mathcal B}^*$. Just is the case of the foregoing example,
where $\zL_1 \zeq de_{m}^*$, $\za= dx_n$ and $b_{\za}=e_n$. }  
\end{remark}

\section{Other results for the linear case}\label{sec-6}

For non-unimodular Lie algebras, from a flat Poisson pair one may construct
another one that is not flat. More exactly:

\begin{proposition}\label{pro-3}
Consider a linear Poisson pair $(\zL,\zL_1 )$ on the dual space ${\mathcal A}^*$
of a non-unimodular Lie algebra ${\mathcal A}$, of dimension $m=2n-1\zmai 5$, and an  
element $p$ of ${\mathcal A}^*$. If $(\zL,\zL_1 )$ is generic and flat at $p$, then every
linear Poisson pair $(\zL,\zL_1 +a\zL(p))$, $a\zpe\mathbb K -\{0\}$, is generic
and non-flat at $p$. 

Therefore, given a non-unimodular Lie algebra ${\tilde{\mathcal A}}$ of odd dimension 
$\zmai 5$, if there exist $\tilde\za,\tilde\zb\zpe{\tilde{\mathcal A}}^*$
such that $(d\tilde\za,d\tilde\zb)$ is generic, then on ${\tilde{\mathcal A}}^*$ there is a 
generic and non-flat linear Poisson pair.  
\end{proposition}

\begin{proof}
Clearly $(\zL,\zL_1 +a\zL(p))$ is generic at $p$. On the other hand, 
let $X$ be the vector field of
lemma \ref{lem-5} and $\za,\za_{a}$ $1$-forms on ${\mathcal A}^*$ Casimir of
$\zL_1$ and $\zL_1 +a\zL(p)$ respectively, both of them non-vanishing at $p$. As
$a\znoi 0$  vectors $\zL(\za,\quad)(p)$ and
$\zL(\za_{a},\quad)(p)$ are linearly independent
(that follows from being generic since $dim{\mathcal A}\zmai 5$.) By lemma \ref{lem-3},
$X(p)$ and $\zL(\za,\quad)(p)$ are linearly dependent because  $(\zL,\zL_1 )$ is flat.
Therefore $X$ and $\zL(\za_{a},\quad)$ are independent at $p$, so about $p$.

For the second assertion it is enough to remark that $(\tilde\zL,\tilde\zL_1 )$, where 
$\tilde\zL$ is the Lie-Poisson structure on ${\tilde{\mathcal A}}^*$ and 
$\tilde\zL_1\zeq d\tilde\zb$, is generic at $\tilde\za$; if $(\tilde\zL,\tilde\zL_1 )$ is
flat at $\tilde\za$ apply the first statement. 
\end{proof}

\begin{example}\label{eje-2}
{\rm Let ${\mathcal A}$ be the Lie algebra of basis $\{e_1 ,\zps,e_5 \}$ given by
$[e_1 ,e_5 ]=e_5$, $[e_2 ,e_3 ]=e_3$, $[e_2 ,e_4 ]=-e_4$ and $[e_i ,e_j ]=0$,
$i<j$, otherwise (that corresponds on ${\mathbb K}^3$ to vector fields
$e_1 =\zpar/\zpar z_1$, $e_2 =\zpar/\zpar z_2$, $e_3=exp(z_2 )(\zpar/\zpar z_3 )$, 
$e_4=exp(-z_2 )(\zpar/\zpar z_3 )$ and $e_5=exp(z_1 )(\zpar/\zpar z_1 )$). Then with 
respect to the coordinates associated to the dual basis $\{e_{1}^* ,\zps,e_{5}^* \}$
one has $$\zL=x_5 {\frac {\zpar} {\zpar x_1 }}\zex{\frac {\zpar} {\zpar x_5 }}
+x_3 {\frac {\zpar} {\zpar x_2 }}\zex{\frac {\zpar} {\zpar x_3 }}
-x_4 {\frac {\zpar} {\zpar x_2 }}\zex{\frac {\zpar} {\zpar x_4 }}.$$ 

On the other hand $e_{1}^* \zex e_{2}^* +e_{3}^* \zex e_{4}^*$ is a $2$-cocycle,
so $(\zL,\zL_1 )$, where $\zL_1 =(\zpar/\zpar x_1 )\zex(\zpar/\zpar x_2 )
+ (\zpar/\zpar x_3 )\zex(\zpar/\zpar x_4 )$, is a Poisson pair. It is easily checked that 
$dx_5$ is a $\zL_1$-Casimir and $(\zL,\zL_1 )$ generic at $p=(0,0,1,0,1)$. 

The vector field $X$ given by lemma \ref{lem-5} equals $\zpar/\zpar x_1$ while
$\zL(dx_5 ,\quad)=-x_5 (\zpar/\zpar x_1 )$, so $(\zL,\zL_1 )$ is flat at $p$ (lemma
\ref{lem-3}.) Nevertheless  $(\zL,\zL_1+a\zL(p))$, $a\zpe{\mathbb K}-\{0\}$, is not
flat (proposition \ref{pro-3}.) For checking this fact directly, observe that 
$dx_5 -adx_2 -a^2 dx_4$ is a Casimir of
$$\zL_1+a\zL(p)={\frac {\zpar} {\zpar x_1 }}\zex{\frac {\zpar} {\zpar x_2 }}
+{\frac {\zpar} {\zpar x_3 }}\zex{\frac {\zpar} {\zpar x_4 }}
+a\zpizq {\frac {\zpar} {\zpar x_1 }}\zex{\frac {\zpar} {\zpar x_5 }}
+{\frac {\zpar} {\zpar x_2 }}\zex{\frac {\zpar} {\zpar x_3 }}\zpder$$
and $$\zL(dx_5 -adx_2 -a^2 dx_4 ,\quad)(p)=-{\frac {\zpar} {\zpar x_1}}
-a{\frac {\zpar} {\zpar x_3}}\, ;$$
now apply lemma \ref{lem-3}.}
\end{example}

For unimodular algebras a more sophisticated construction is needed. Until the end of this
section  ${\mathcal A}$ will denote a Lie algebra of dimension $m=2n-1\zmai 3$. To each
element $\za\zpe{\mathcal A}^*$ such that $(d\za)^{n-1}\znoi 0$ one associates a Lie
subalgebra ${\mathcal A}_{\za}$ as follows:
\begin{enumerate}
\item  If $\za\zex(d\za)^{n-1}\znoi 0$ then ${\mathcal A}_{\za}=0$.
\item  If $\za\zex(d\za)^{n-1}=0$ then there exists $v\zpe{\mathcal A}$ such that 
$i_v d\za=-\za$; moreover if $i_w d\za=-\za$ for another $w\zpe{\mathcal A}$ then
$w-v\zpe Kerd\za$. By definition ${\mathcal A}_{\za}$ will be the $2$-dimensional vector
space spanned by $v$ and $Kerd\za$, which is a Lie subalgebra because one has  
$L_v d\za=-d\za$ that implies $[v,Kerd\za]\zco Kerd\za$.

In this case $v$ will be named a {\it Hamiltonian of} $\za$. 
\end{enumerate}

\begin{lemma}\label{lem-7}
Suppose ${\mathcal A}$ unimodular. Consider $\za\zpe{\mathcal A}^*$ such that 
$(d\za)^{n-1}\znoi 0$. Then ${\mathcal A}_{\za}$ is either non-abelian or zero.
\end{lemma}

\begin{proof}    
Assume $dim{\mathcal A}_{\za}=2$. Let $v$ a Hamiltonian of $\za$ and $u$ a basis of
$Kerd\za$. Take $\zt\zpe\zL^{m}{\mathcal A}^* -\{0\}$. Then $i_u \zt=c(d\za)^{n-1}$
with $c\znoi 0$.

Since ${\mathcal A}$  is unimodular, $L_v \zt=0$; now if $[u,v]=0$ one has
$$0=i_u L_v \zt=L_v (i_u \zt)=cL_v ((d\za)^{n-1})=-c(n-1)(d\za)^{n-1}\znoi 0$$
{\it contradiction}. 
\end{proof}

For the purpose of this work, a couple $(\za,\zb)\zpe{\mathcal A}^* \zpor{\mathcal A}^*$ 
will be called {\it generic} if $(d\za,d\zb)$ is generic and  ${\mathcal A}_{(s\za+t\zb)}$ is
non-abelian or zero for all $(s,t)\zpe{\mathbb C}^2 -\{0\}$; note that if
${\mathcal A}_{\zb}$ is non-abelian or zero, it suffices to check the property for
every  ${\mathcal A}_{(\za+t\zb)}$, $t\zpe{\mathbb C}$. 
Even when ${\mathcal A}$ is a real Lie algebra, this definition is meaningful by complexifying
it. On the other hand, by lemma \ref{lem-7}, when ${\mathcal A}$ is unimodular 
$(\za,\zb)$ is generic if $(d\za,d\zb)$ is generic. 

The next step is to construct a new Lie algebra ${\mathcal B}_{\mathcal A}$, called {\it 
the secondary algebra} (of  ${\mathcal A}$). Set
 ${\mathcal B}_{\mathcal A}={\mathcal A}\zpor{\mathcal A}\zpor\mathbb K$ endowed
with the bracket
$$[(v,v',s),(w,w',t)]=([v,w],[v,w']-[w,v']+tv'-sw',0).$$

This new algebra is just the extension of ${\mathcal A}$ by ${\mathcal A}$, 
this second time regarded as an abelian ideal, by means of the adjoint representation
plus more one dimension for taking into account $I\zdp{\mathcal A}\zfl{\mathcal A}$.
Clearly ${\mathcal B}_{\mathcal A}$ is not unimodular and its dimension equals
$2m+1$.

If $\{{\tilde e}_1 ,\zps, {\tilde e}_m\}$ is a basis of ${\mathcal A}$,
$[{\tilde e}_i ,{\tilde e}_j ]=\zsu_{k=1}^{m}c_{ij}^{k}{\tilde e}_k$, $i,j=1,\zps,m$, and
we set $e_r =({\tilde e}_r ,0,0)$, $f_r =(0,{\tilde e}_r ,0)$, $r=1,\zps,m$ and 
$e=(0,0,1)$, then $\{e_1 ,\zps,e_m ,f_1 ,\zps,f_m ,e\}$ is a basis of 
${\mathcal B}_{\mathcal A}$  and
$[e_i ,e_j ]=\zsu_{k=1}^{m}c_{ij}^{k}e_k$,
$[e_i ,f_j ]=-[f_j ,e_i ]=\zsu_{k=1}^{m}c_{ij}^{k}f_k$,
$[f_j ,e]=-[e,f_j ]=f_j$, $i,j=1,\zps,m$, while the other brackets vanish.

In coordinates  
$(x,y,z)=(x_1 ,\zps,x_m ,y_1 ,\zps,y_m ,z)$ associated to the dual basis
$\{e_{1}^* ,\zps,e_{m}^* ,f_{1}^* ,\zps,f_{m}^* ,e^*\}$ the Lie-Poisson 
structure on  ${\mathcal B}_{\mathcal A}^{*}$ writes: 
$$\zL=\zsu_{1\zmei i<j\zmei m}\zsu_{k=1}^{m}c_{ij}^{k}x_k
{\frac {\zpar} {\zpar x_i}}\zex{\frac {\zpar} {\zpar x_j}}\hskip 4 truecm$$
$$\hskip 3 truecm +\zsu_{i,j,k=1}^{m}c_{ij}^{k}y_k
{\frac {\zpar} {\zpar x_i}}\zex{\frac {\zpar} {\zpar y_j}}
+\zpizq\zsu_{k=1}^{m}y_k {\frac {\zpar} {\zpar y_k}}\zpder
\zex{\frac {\zpar} {\zpar z}}.$$
\smallskip

\begin{proposition}\label{pro-4}
Let ${\mathcal A}$ be a Lie algebra of dimension $m=2n-1\zmai 3$. If there exists a
generic couple $(\za,\zb)\zpe{\mathcal A}^* \zpor{\mathcal A}^*$ such that 
$\zb\zex(d\zb)^{n-1}\znoi 0$, then on  ${\mathcal B}_{\mathcal A}^*$ there is 
some non-flat generic linear Poisson pair.
\end{proposition}

\begin{proof}
Observe that the the center of ${\mathcal A}$ is zero because it is included in
$Kerd\za\zin Kerd\zb$, which is zero since $(d\za,d\zb)$ is generic. Regarding $(d\za,d\zb)$
as a couple of bi-vectors on ${\mathcal A}^*$ and taking into account that Casimirs of
$d\za+td\zb$, $t\zpe \mathbb K$, correspond to elements of $Ker(d\za+td\zb)$ show
the existence of a polynomial curve $\zg(t)=\zsu_{j=0}^{n}t^j a_j$ in ${\mathcal A}$,
with $a_0 ,\zps,a_n$ linearly independent, such 
that $\zg(t)$,  $t\zpe \mathbb K$, is a basis
of  $Ker(d\za+td\zb)$ (see \cite{TU1,TU2}.) Now choose $\zr\zpe {\mathcal A}^*$ such
that $\zr(a_0 )=1$ and $\zr(a_k )=0$, $k=1,\zps,n$; then $\zr(\zg(t))\znoi 0$ for
every $t\zpe \mathbb C$ (in the real case complexify ${\mathcal A}$.)  

On the other hand we identify  ${\mathcal B}_{\mathcal A}^*$ and 
${\mathcal A}^* \zpor{\mathcal A}^* \zpor(\mathbb K )^*$ in the obvious way. Set
$\zL_1 =\zL(0,\zb,0)$; we shall prove that $(\zL,\zL_1 )$ is generic and non-flat
at $(\zr,\za,0)$. 

First let us check that $rank\zL_1 =2m$. As $\zb\zex (d\zb)^{n-1}\znoi 0$, there exists
a basis $\{{\tilde e}_1 ,\zps, {\tilde e}_m\}$ of ${\mathcal A}$ such that 
$\zb={\tilde e}_{1}^*$ and $d\zb=-\zsu_{j=1}^{n-1}{\tilde e}_{2j}^* 
\zex{\tilde e}_{2j+1}^*$. Then in coordinates $(x,y,z)$ of 
${\mathcal B}_{\mathcal A}^*$ one has
$$\zL_1 =\zsu_{j=1}^{n-1}\zpizq
{\frac {\zpar} {\zpar x_{2j}}}\zex{\frac {\zpar} {\zpar y_{2j+1}}}
-{\frac {\zpar} {\zpar x_{2j+1}}}\zex{\frac {\zpar} {\zpar y_{2j}}}\zpder
+{\frac {\zpar} {\zpar y_{1}}}\zex{\frac {\zpar} {\zpar z}}$$ 
whose rank equals $2m$. 

Observe that $dx_1$ is a Casimir of $\zL_1$, which corresponds to 
$e_1 \zpe {\mathcal B}_{\mathcal A}$. As $rank([e_1 ,\quad])\zmai 2$ because
${\mathcal A}$ has trivial center, from lemmas \ref{lem-3} and \ref{lem-5} and
remark \ref{rem-5} follows the non-flatness of $(\zL,\zL_1 )$ at $(\zr,\za,0)$ provided
that it is generic. 

Since $rank\zL_1 =2m$, for verifying that $(\zL,\zL_1 )$ is generic at $(\zr,\za,0)$ it
suffices to show that every $(\zL+t\zL_1 )(\zr,\za,0)$, $t\zpe \mathbb C$, has rank 
$2m$, which is equivalent to see that $\zL(\zr,\za+t\zb,0)$, $t\zpe \mathbb C$, has
rank $2m$. Set $\zt=\za+t\zb$; by complexifying ${\mathcal A}$ if necessary, we may
suppose that it is a complex Lie algebra without loss of generality. If
 $\zt\zex(d\zt)^{n-1}\znoi 0$ by considering a second basis 
$\{{\tilde e}_1 ,\zps, {\tilde e}_m\}$  of ${\mathcal A}$ (denoted as the first one for sake 
of simplicity) such that $\zt={\tilde e}_{1}^*$
and $d\zt=-\zsu_{j=1}^{n-1}{\tilde e}_{2j}^* \zex{\tilde e}_{2j+1}^*$, in coordinates
$(x,y,z)$ on ${\mathcal B}_{\mathcal A}^*$ one has
$$\zL(\zr,\zt,0) =\zsu_{1\zmei i<j\zmei m}a_{ij}
{\frac {\zpar} {\zpar x_i}}\zex{\frac {\zpar} {\zpar x_j}}\hskip 6truecm$$
$$+\zsu_{j=1}^{n-1}\zpizq
{\frac {\zpar} {\zpar x_{2j}}}\zex{\frac {\zpar} {\zpar y_{2j+1}}}
-{\frac {\zpar} {\zpar x_{2j+1}}}\zex{\frac {\zpar} {\zpar y_{2j}}}\zpder
+{\frac {\zpar} {\zpar y_{1}}}\zex{\frac {\zpar} {\zpar z}}\hskip 1truecm$$
$$\hskip 1truecm =\zsu_{j=1}^{n-1}\zpizq
{\frac {\zpar} {\zpar x_{2j}}}\zex v_{2j+1}
-{\frac {\zpar} {\zpar x_{2j+1}}}\zex v_{2j}\zpder
+{\frac {\zpar} {\zpar y_{1}}}\zex{\frac {\zpar} {\zpar z}}\, ,$$
\vskip .3truecm
\noindent where $v_k =(\zpar/\zpar y_{k})+\zsu_{r=1}^{m}b_{kr}(\zpar/\zpar x_{r})$,
$k=2,\zps,m$ for suitable scalars $b_{kr}$, whose rank equals $2m$. 

Now assume  $\zt\zex(d\zt)^{n-1}=0$. As $(d\zt)^{n-1}\znoi 0$ since $(d\za,d\zb)$ is
generic and $\zr_{\zbv Kerd\zt}\znoi 0$ because $\zr(\zg(t))\znoi 0$, there exists a basis
$\{{\tilde e}_1 ,\zps, {\tilde e}_m\}$  of ${\mathcal A}$ (denoted as the first and second
ones) such that $\zt={\tilde e}_{2}^*$, $d\zt=-\zsu_{j=1}^{n-1}{\tilde e}_{2j}^* 
\zex{\tilde e}_{2j+1}^*$ and $\zr={\tilde e}_{1}^*$. Thus 
$\{{\tilde e}_{1},{\tilde e}_{3}\}$ is a basis of ${\mathcal A}_{\zt}$,
$[{\tilde e}_{1},{\tilde e}_{3}]=c_{13}^{1}{\tilde e}_{1}$ with $c_{13}^{1}\znoi 0$ 
since ${\mathcal A}_{\zt}$ is non-abelian, and in coordinates $(x,y,z)$ on
${\mathcal B}_{\mathcal A}^*$ one has $x_1 (\zr,\zt,0)=y_2 (\zr,\zt,0)=1$ while
the remainder coordinates of $(\zr,\zt,0)$ vanish. Therefore
\vskip .3truecm\noindent   
$$\zL(\zr,\zt,0) =\zsu_{1\zmei i<j\zmei m}c_{ij}^{1}
{\frac {\zpar} {\zpar x_i}}\zex{\frac {\zpar} {\zpar x_j}}\hskip 6truecm$$
$$\hskip 1truecm +\zsu_{j=1}^{n-1}\zpizq
{\frac {\zpar} {\zpar x_{2j}}}\zex{\frac {\zpar} {\zpar y_{2j+1}}}
-{\frac {\zpar} {\zpar x_{2j+1}}}\zex{\frac {\zpar} {\zpar y_{2j}}}\zpder
+{\frac {\zpar} {\zpar y_{2}}}\zex{\frac {\zpar} {\zpar z}}$$ 
$$=c_{13}^{1}{\frac {\zpar} {\zpar x_1}}\zex{\frac {\zpar} {\zpar x_3}}
+{\frac {\zpar} {\zpar x_2}}\zex w_{3}
-\zpizq{\frac {\zpar} {\zpar x_{3}}}
+{\frac {\zpar} {\zpar z}}\zpder\zex{\frac {\zpar} {\zpar y_2}}\hskip 2truecm$$
$$\hskip 2truecm +\zsu_{j=2}^{n-1}\zpizq {\frac {\zpar} {\zpar x_{2j}}}\zex w_{2j+1}
-{\frac {\zpar} {\zpar x_{2j+1}}}\zex w_{2j}\zpder ,$$
\vskip .3truecm\noindent
 where $w_k =(\zpar/\zpar y_{k})+\zsu_{r=1}^{m}{\tilde b}_{kr}(\zpar/\zpar x_{r})$,
$k=3,\zps,m$ for suitable scalars ${\tilde b}_{kr}$. Clearly its rank equals $2m$ 
since  $c_{13}^{1}\znoi 0$.
\end{proof}

\begin{proposition}\label{pro-5}
Consider an unimodular Lie algebra ${\mathcal A}$ of dimension $m=2n-1\zmai 3$; assume 
that this algebra possesses some contact form. If there exist $\za,\zb\zpe{\mathcal A}^*$
such that $(d\za,d\zb)$ is generic then on  ${\mathcal B}_{\mathcal A}^*$ there exists 
some non-flat generic linear Poisson pair. 
\end{proposition}

\begin{proof}
The set of contact forms is open and dense in ${\mathcal A}^*$. As to be generic is an
open property, there exists $\zb'\zpe{\mathcal A}^*$ such that $(d\za,d\zb')$ is
generic and $\zb'\zex(d\zb')^{n-1}\znoi 0$. Now apply lemma \ref{lem-7} and
proposition \ref{pro-4}. 
\end{proof}

\begin{example}\label{eje-3}
{\rm Proposition \ref{pro-4} (proposition \ref{pro-5} is just a particular case of the foregoing
one) can be applied to a $3$-dimensional Lie algebra ${\mathcal A}$ if and only if either it is 
simple or there exists a basis $\{e_1 ,e_2 ,e_3 \}$ such that $[e_1 ,e_2 ]=e_3$,
$[e_1 ,e_3 ]=ae_2 +be_3$ with $a\znoi 0$ and $[e_2 ,e_3 ]=0$. Indeed, the simple case
is obvious; therefore assume solvable ${\mathcal A}$. 

If proposition \ref{pro-4} applies, then ${\mathcal A}$ possess contact forms and its center
is trivial (see the beginning of the proof of proposition \ref{pro-4}.) In this case a 
computation shows the existence of this basis (as its center equals zero ${\mathcal A}$
contains a $2$-dimensional abelian ideal ${\mathcal A}_0$ such that 
$[v,\quad]\zdp{\mathcal A}_0 \zfl{\mathcal A}_0$ is an isomorphism for any
$v\znope{\mathcal A}_0$; the existence of contact forms implies that
$[v,\quad]\zdp{\mathcal A}_0 \zfl{\mathcal A}_0$ is never multiple of identity.) 

Conversely, when a such basis exists it suffices to set $\za=e_{2}^*$ and 
$\zb=e_{3}^*$.}
\end{example}

\begin{example}\label{eje-4}
{\rm Consider the Lie algebra ${\mathcal A}$ of basis $\{e_1 ,\zps,e_{2n-1}\}$,
$n\zmai 2$, given by $[e_{2j-1},e_{2j}]=-e_{2j}$ and $[e_{2j-1},e_{2n-1}]=-ae_{2n-1}$
with $a\znoi 0$, $j=1,\zps,n-1$, and $[e_{i},e_{r}]=0$, $i<r$, otherwise (this algebra
corresponds to consider vector fields $e_{2j-1}=\zpar/\zpar x_j$  and 
$e_{2j}=exp(-x_j )(\zpar/\zpar x_n )$, $j=1,\zps,n-1$, and 
$e_{2n-1}=exp(-\zsu_{k=1}^{n-1}ax_{k})(\zpar/\zpar x_n )$ on ${\mathbb K}^n$.)
Set $\za=\zsu_{j=1}^{n-1}e_{2j}^*$ and 
$\zb=\zsu_{j=1}^{n-1}a_j e_{2j}^* +e_{2n-1}^*$ where $a_1 ,\zps,a_{n-1}$ are
distinct and non-vanishing scalars.

Then
$$(d\za+td\zb)^{n-1}=\zpizq\zsu_{j=1}^{n-1}(a_j t+1)e_{2j-1}^{*}\zex e_{2j-1}^{*}                
+at(\zsu_{j=1}^{n-1}e_{2j-1}^{*})\zex e_{2n-1}^{*} \zpder^{n-1}$$
$$=(n-1)!\zpizq\zpr_{j=1}^{n-1}(a_j t+1)\zpder e_{1}^{*}\zex\zps\zex e_{2n-2}^{*}$$ 
$$+(n-1)!at\zsu_{k=1}^{n-1}\zcizq\zpizq\zpr_{j=1;j\znoi k}^{n-1}(a_j t+1)\zpder
e_{1}^{*}\zex\zps\zex{\widehat e_{2k}^{*}}\zex\zps\zex e_{2n-1}^{*}\zcder$$
so non-zero for any $t\zpe\mathbb C$.

On the other hand
$$(\za+t\zb)\zex(d\za+td\zb)^{n-1}=(n-1)!(a[1-n]+1)t\zpizq\zpr_{j=1}^{n-1}(a_j 
t+1)\zpder e_{1}^{*}\zex\zps\zex e_{2n-1}^{*}$$
while $\zb\zex(d\zb)^{n-1}=(n-1)!(a[1-n]+1)a_1 \zpu\zpu\zpu a_{n-1}
e_{1}^{*}\zex\zps\zex e_{2n-1}^{*}$.

Let us suppose $a\znoi (n-1)^{-1}$. Then $\zb$ is a contact form and $(d\za,d\zb)$ is
generic. Thus if $a=-1$, as ${\mathcal A}$ is unimodular, one may apply proposition 
\ref{pro-5} to $(\za,\zb)$. In the general case we have to examine algebras 
${\mathcal A}_{(\za+t\zb)}$, $t\zpe\mathbb C$. They are of dimension two just when
$t=0, -a_{1}^{-1} ,\zps,-a_{n-1}^{-1}$. 

A basis of ${\mathcal A}_{\za}$, which corresponds to $t=0$,
is $\{\zsu_{j=1}^{n-1}e_{2j-1},e_{2n-1}\}$, so this
algebra is not abelian. Now suppose $t=-a_{1}^{-1}$ (the remainder cases are similar); 
then a basis of ${\mathcal A}_{(\za-a_{1}^{-1}\zb)}$ is
$\{e_2 ,(a^{-1}-n+2)e_1 +\zsu_{j=2}^{n-1}e_{2j-1} \}$ and it is non-abelian if and
only if $a^{-1}-n+2\znoi 0$, that is $a\znoi (n-2)^{-1}$. 

Summing up, proposition \ref{pro-4} may be applied just when 
$a\znoi 0, (n-1)^{-1}, (n-2)^{-1}$.}
\end{example}

\begin{example}\label{eje-5}
{\rm Let ${\mathcal A}$ be the truncated Lie algebra of dimension $m=2n-1\zmai 3$ of
example \ref{eje-1}. Recall that the couple $(de_{m}^* ,de_{m-1}^* )$ was generic,
so $(d\za,d\zb)$ where $\za=e_{m}^*$ and $\zb=e_{m}^* +e_{m-1}^*$ is
generic too. By lemma \ref{lem-6} $\zb$ is a contact form and 
${\mathcal A}_{(\za+t\zb)}=\{0\}$ unless $t=0,-1$. Therefore to see that $(\za,\zb)$
is generic one has to show that ${\mathcal A}_{e_{m}^*}$ and 
${\mathcal A}_{e_{m-1}^*}$ are not abelian.

But $\{e_1 ,e_n \}$ is a basis of ${\mathcal A}_{e_{m}^*}$ and 
$[e_1 ,e_n]=(n-1)e_n$ while $\{e_1 ,e_m \}$ is a basis of ${\mathcal A}_{e_{m-1}^*}$ 
and $[e_1 ,e_m]=(m-1)e_m$, so $(\za,\zb)$ is generic and proposition \ref{pro-4} may
be applied to it.}
\end{example}

\section{The special affine algebra}\label{sec-7}
In this section we show that proposition \ref{pro-5} may be applied to this Lie algebra.
Let $V$ be a real or complex vector space of dimension $n\zmai 2$ and 
 ${\mathcal Aff}(V)$ the affine algebra of $V$, which can be regarded too like the algebra
of polynomial vector fields on $V$ of degree $\zmei 1$. Recall that 
${\mathcal Aff}(V)= {\mathcal I}\zdi sl(V)\zdi V$ where ${\mathcal I}$ consists of linear
vector fields multiples of identity, $sl(V)$ is the special linear algebra of $V$ and $V$ the 
ideal of constant vector fields.
On the other hand, the dual space ${\mathcal Aff}(V)^*$ will be identify to  
${\mathcal I}^* \zdi sl(V)^* \zdi V^*$ in the obvious way.

Denote by $kil$ the Killing form of $sl(V)$, which is non-degenerate; therefore if one sets
$\za_g =kil(g,\quad)$, $g\zpe sl(V)$, then $g\zpe sl(V)\zfl\za_g \zpe sl(V)^*$ is
an isomorphism of vector spaces. Moreover, given $g,h\zpe sl(V)$ then 
$(d\za_g )(h,\quad)=-\za_{[g,h]}$; thereby $(d\za_g )(h,\quad)=0$ if and only if
 $[g,h]=0$.  

For any $g\zpe gl(V)$ and $\zt\zpe V^*$ the subspace spanned by $g,\zt$ means that
spanned by $\zt$ and the dual endomorphism $g^*$. One will need the following results:

\begin{lemma}\label{lem-8}
On a real or complex vector space $E=E_1 \zdi E_2$, of even dimension, consider a couple
of $2$-forms $\zl,\zl_1$ such that $Ker\zl\zcco E_2$, $\zl_{1\zbv E_1}=0$ and
$\zl_{1\zbv E_2}=0$. One has:

\noindent (a) If $\zl_{1\zbv Ker\zl}$ is symplectic then $\zl+\zl_1$ is symplectic too.

\noindent (b) If the corank of $\zl_{1\zbv Ker\zl}$ equals two and 
$dim (E_1 \zin Ker(\zl_{1\zbv Ker\zl}))\zmai 1$, then the corank of
$\zl+\zl_1$ equals two.
\end{lemma}

\begin{lemma}\label{lem-9}
For any $n\zmai 2$ one may find real numbers $a_1 ,\zps,a_n$, $b_1 ,\zps,b_n$,
 $c_1 ,\zps,c_n$ satisfying:

\noindent (I) $a_i \znoi a_j$, $b_i \znoi b_j$ and $c_i \znoi c_j$ whenever $i\znoi j$; 
moreover no $c_i$, $i=1,\zps,n$, vanishes.

\noindent (II) $\zsu_{i=1}^{n}a_i =\zsu_{i=1}^{n}b_i =0$.

\noindent (III) For every $t\zpe{\mathbb C}-\{0\}$, at least $n-1$ elements of 
the family $a_1 +tb_1  ,\zps,a_n +tb_n$ are different.
Moreover is a such family includes two equal elements then $1+tc_i \znoi 0$ for every
 $i=1,\zps,n$. 
\end{lemma}

\begin{proof}
Consider a family $a'_1 ,\zps,a'_n$ of distinct rational numbers and a second one
$b'_1 ,\zps,b'_n$ of rationally independent real numbers. Set 
$t_{ij}=(a'_i -a'_j )(b'_i -b'_j )^{-1}$, $i\znoi j$. Then $a'_i +tb'_i =a'_j +tb'_j$ if 
and only if $t=-t_{ij}$.

Suppose $a'_i +tb'_i =a'_j +tb'_j$ and $a'_k+tb'_k =a'_r +tb'_r$ for some $i<j$ and
$k<r$ with $(i,j)\znoi (k,r)$. Then $t=-t_{ij}=-t_{kr}$, that is $t_{ij}=t_{kr}$, and an
elementary computation shows that $b'_1 ,\zps,b'_n$ are not rationally independent;
therefore (III) holds for these two families.

Observe that (III) holds too for $a'_1 +a,\zps,a'_n +a$ and $b'_1 +b,\zps,b'_n +b$ 
whatever $a,b$ are, because each $t_{ij}$ does not change. In other words, setting 
$a=-\zsu_{i=1}^{n}(a'_i /n)$ and $b=-\zsu_{i=1}^{n}(b'_i /n)$ shows the existence 
of two families $a_1 ,\zps,a_n$ and $b_1 ,\zps,b_n$ satisfying (I), (II) and (III).   

Finally, choose $c_1 ,\zps,c_n \zpe{\mathbb R}-(\{0\}\zun\{t_{ij}^{-1} \zbv 
1\zmei i<j\zmei n\})$ that are distinct among them.
\end{proof}

\begin{proposition}\label{pro-6}
Given a $n$-dimensional, $n\zmai 2$, real or complex vector space consider $g\zpe sl(V)$
and $\zt\zpe V^*$ and regard $\za_g +\zt$ like an element of  ${\mathcal Aff}(V)^*$.
Assume diagonalizable $g$. One has:

\noindent (a) If the eigenvalues of $g$ are distinct and $g,\zt$ span $V^*$, then 
$d(\za_g +\zt)$ is symplectic.

\noindent (b) If, at least, $n-1$ eigenvalues are different and $g,\zt$ span a vector 
subspace of dimension $n-1$, then 
$rank(d(\za_g +\zt))=n^2 +n-2=dim{\mathcal Aff}(V)-2$.  
Moreover $Ker(d(\za_g +\zt))$ is not included in $sl(V)\zdi V$. 
\end{proposition}

\begin{proof}
Set $E= {\mathcal Aff}(V)$, $E_1 ={\mathcal I}\zdi sl(V)$, $E_2 =V$, $\zl=d\za_g$
and $\zl_1 =d\zt$.

\noindent (a) In this case there is a basis $\{v_1 ,\zps,v_n \}$ of $V$ such that
$g=\zsu_{j=1}^{n}a_j v_j \zte v_{j}^{*}$, where $a_i \znoi a_j$ if $i\znoi j$, and
$\zt=\zsu_{j=1}^{n} v_{j}^{*}$. Then 
$\{v_1 \zte v_{1}^{*},\zps,v_n \zte v_{n}^{*},v_1 ,\zps,v_n \}$ is a a basis of
$Kerd\za_g$ and
$$d\zt_{\zbv Kerd\za_g}=\zmm\zpizq\zsu_{j=1}^{n}(v_j \zte 
v_{j}^{*})^{*}\zex  v_{j}^{*}\zpder_{\zbv Kerd\za_g}$$ 
where $\{\{v_i \zte  v_{j}^{*}\},i,j=1,\zps,n,v_1 ,\zps, v_n \}$ is the basis of
${\mathcal Aff}(V)$ associated to $\{v_1 ,\zps,v_n \}$ and 
$\{\{(v_i \zte  v_{j}^{*})^{*}\},i,j=1,\zps,n,v_{1}^* ,\zps, v_{n}^* \}$ its dual basis.
Now apply (a) of lemma \ref{lem-8}.

\noindent (b) This time there are two possible cases. First assume that all eigenvalues of 
$g$ are distinct; then there exists a basis  $\{v_1 ,\zps,v_n \}$ of $V$ such that
$g=\zsu_{j=1}^{n}a_j v_j \zte v_{j}^{*}$, with $a_i \znoi a_j$ if $i\znoi j$, and
$\zt=\zsu_{j=1}^{n-1} v_{j}^{*}$.

On the other hand $Kerd\za_g$ is the same as before while
$$d\zt_{\zbv Kerd\za_g}=\zmm\zpizq\zsu_{j=1}^{n-1}(v_j \zte 
v_{j}^{*})^{*}\zex  v_{j}^{*}\zpder_{\zbv Kerd\za_g}$$
and it suffices applying (b) of lemma \ref{lem-8} for computing the rank.

Now suppose that two eigenvalues are equal; in this case there exists  a basis 
 $\{v_1 ,\zps,v_n \}$ of $V$ such that $g=\zsu_{j=1}^{n-2}a_j v_j \zte v_{j}^{*}
+a_{n-1}(v_{n-1} \zte v_{n-1}^{*}+v_{n} \zte v_{n}^{*})$, with $a_i \znoi a_j$ 
if $i\znoi j$, and $\zt=\zsu_{j=1}^{n-1} v_{j}^{*}$. Then 
$\{v_1 \zte v_{1}^{*},\zps,v_{n-2} \zte v_{n-2}^{*},\{v_{k} \zte v_{r}^{*}\},
k,r=n-1,n,v_1 ,\zps,v_n \}$ is a basis of $Kerd\za_g$ and $d\zt_{\zbv Kerd\za_g}$
equals 
$$\zmm\zpizq\zsu_{j=1}^{n-2}(v_j \zte 
v_{j}^{*})^{*}\zex  v_{j}^{*}+(v_{n-1} \zte  v_{n-1}^{*})^{*}\zex  v_{n-1}^{*}
+(v_{n-1} \zte  v_{n}^{*})^{*}\zex  v_{n}^{*}\zpder_{\zbv Kerd\za_g}$$
and it is enough to apply (b) of lemma \ref{lem-8} for computing the rank.

Finally note that in both cases $v_{n} \zte  v_{n}^{*}$ belongs to $Ker(d\za_g +\zt)$.
\end{proof}

\begin{example}\label{eje-6}
{\rm Let ${\mathcal Aff}_{0}(V)$ be the special affine algebra of $V$, that is
${\mathcal Aff}_{0}(V)=sl(V)\zdi V$. Consider a basis $\{v_1 ,\zps,v_n \}$ of $V$ and
scalars $a_1 ,\zps,a_n$, $b_1 ,\zps,b_n$,  $c_1 ,\zps,c_n$ as in lemma \ref{lem-9}. Set
$g=\zsu_{j=1}^{n}a_j v_j \zte v_{j}^{*}$, $h=\zsu_{j=1}^{n}b_j v_j \zte v_{j}^{*}$,
$\zt=\zsu_{j=1}^{n} v_{j}^{*}$ and $\zm=\zsu_{j=1}^{n} c_j v_{j}^{*}$. Let
${\tilde\za}=\za_g +\zt$ and ${\tilde\zb}=\za_h +\zm$ that are $1$-forms on 
${\mathcal Aff}(V)$; then $\za={\tilde\za}_{\zbv {\mathcal Aff}_{0}(V)}$ and 
 $\zb={\tilde\zb}_{\zbv {\mathcal Aff}_{0}(V)}$ are contact forms on 
 ${\mathcal Aff}_{0}(V)$. Indeed, we prove it for $\za$ the other case is analogous.
By (a) of proposition \ref{pro-6} $d\tilde\za$ is symplectic, so there is 
$z\zpe{\mathcal Aff}(V)$ such that $i_z d\tilde\za=\tilde\za$, which implies that
$L_z ((d\tilde\za)^{n(n+1)/2})\znoi 0$ that is to say $z\znope{\mathcal Aff}_{0}(V)$. But
$z$ is a basis of the kernel of ${\tilde\za}\zex(d{\tilde\za})^{(n(n+1)/2)-1}$, hence its
restriction to ${\mathcal Aff}_{0}(V)$, which equals ${\za}\zex(d{\za})^{(n(n+1)/2)-1}$,
is a volume form.

Moreover $(d\za,d\zb)$ is generic. Let us see it. Clearly $rank(d\za)$ and  $rank(d\zb)$
equal $dim{\mathcal Aff}_{0}(V)-1$, so it suffices to show that the rank of $d\za+td\zb$,
$t\zpe{\mathbb C}-\{0\}$, is maximal. If $d(\tilde\za+t\tilde\zb)$ is symplectic reason as 
before. If not, taking into account that $\tilde\za+t\tilde\zb=\za_{g+th}+(\zt+t\zm)$, 
$g+th=\zsu_{j=1}^{n}(a_j +tb_j) v_j \zte v_{j}^{*}$ and
$\zt+t\zm=\zsu_{j=1}^{n}(1+t c_j) v_{j}^{*}$, by proposition \ref{pro-6} and lemma
\ref{lem-9} we have two cases: 
\begin{enumerate}
\item The family $a_1 +tb_1  ,\zps,a_n +tb_n$ just includes two equal elements but no
$1+tc_i$, $i=1,\zps,n$, vanishes.
\item All  $a_1 +tb_1  ,\zps,a_n +tb_n$ are distinct but one element of the family 
$1+tc_1  ,\zps,1+tc_n$ vanishes.  
\end{enumerate}

In both cases, by (b) of proposition \ref{pro-6}, $rank(d(\tilde\za+t\tilde\zb))=n^2 +n-2$
and $Kerd(\tilde\za+t\tilde\zb)\znoco{\mathcal Aff}_{0}(V)$, so
$rank(d(\za+t\zb))=rank(d(\tilde\za+t\tilde\zb))=dim{\mathcal Aff}_{0}(V)-1$.

Summing up, one may apply proposition \ref{pro-5} to ${\mathcal Aff}_{0}(V)$ and
$(\za,\zb)$.}
\end{example}

\begin{example}\label{eje-7}
{\rm Consider $\{v_ ,\zps,v_n \}$, $a_1 ,\zps,a_n$, $b_1 ,\zps,b_n$,  $c_1 ,\zps,c_n$ 
as in the foregoing example. Let $a$ be any scalar. Set 
${\mathcal A}(V,a)={\mathcal Aff}(V)\zdi \mathbb K$ endowed with the bracket defined
below. Regard ${\mathcal Aff}(V)$, $\mathbb K$ as subsets of ${\mathcal A}(V,a)$ and 
${\mathcal Aff}(V)^*$, ${\mathbb K}^*$ like subsets of ${\mathcal A}(V,a)^*$ in the
obvious way. Let $e$ be the unit of $\mathbb K$ seen in ${\mathcal A}(V,a)$ and $e^*$ 
the element of ${\mathcal A}(V,a)^*$ given by $e^* (e)=1$, 
$e^* ({\mathcal Aff}(V))=0$. On the other hand 
 let $id\zpe{\mathcal Aff}(V)$ be the morphism identity
of $V$, that is $id=\zsu_{j=1}^{n} v_j \zte v_{j}^{*}$. Now we define a structure of Lie 
algebra on  ${\mathcal A}(V,a)$, for which ${\mathcal Aff}(V)$ is a subalgebra, by putting 
$[id,e]=-ae$ and $[sl(V)\zdi V,e]=0$.

Set $\za_1 =\tilde\za$ and $\zb_1 =\tilde\zb +e^*$ where $\tilde\za,\tilde\zb$, defined in
the preceding example, are now regarded as elements of ${\mathcal A}(V,a)^*$. Note that
$de^* =(a/n)\zsu_{j=1}^{n}(v_j \zte v_{j}^{*})^* \zex e^*$, so
$$(d\zb_1 )^{n(n+1)/2}=(d{\tilde\zb})^{n(n+1)/2}\hskip6truecm$$
$$+aC\zpizq\zsu_{j=1}^{n}(v_j \zte v_{j}^{*})^*\zpder
\zex(d{\tilde\zb})^{(n(n+1)/2)-1}\zex e^* \znoi 0$$
where $C$ is non-zero constant, while
$$\zb_1 \zex(d\zb_1 )^{n(n+1)/2}=(aC'+1)(d{\tilde\zb})^{n(n+1)/2}\zex e^*$$
where $C'$ is another constant (perhaps zero). As $d\tilde\zb$ is symplectic on 
${\mathcal Aff}(V)$ it  follows that $\zb_1$ is a contact form on ${\mathcal A}(V,a)$ if 
and only if $aC'+1\znoi 0$.

On the other hand $(d\za_1 ,d\zb_1 )$ is generic if $a\znoi 0$. Indeed, since $d\tilde\za$
and $d\tilde\zb$ are symplectic on ${\mathcal Aff}(V)$ it is enough to show that 
$d\za_1 +td\zb_1$, $t\zpe{\mathbb C}-\{0\}$, has maximal rank. If 
$d\tilde\za+td\tilde\zb$ is symplectic on ${\mathcal Aff}(V)$ it is clear. Otherwise by 
proposition \ref{pro-6} the rank of $d\tilde\za+td\tilde\zb$ on ${\mathcal Aff}(V)$ equals 
$n^2 +n-2$ and its kernel is not
included in $sl(V)\zdi V$. But, always in ${\mathcal Aff}(V)$, 
$sl(V)\zdi V=Ker(\zsu_{j=1}^{n}(v_j \zte v_{j}^{*})^* )$. Since $d\za_1 +td\zb_1$
equals $d\tilde\za+td\tilde\zb$ regarded on  ${\mathcal A}(V,a)$  plus 
$(at/n)\zsu_{j=1}^{n}(v_j \zte v_{j}^{*})^* \zex e^*$,
 it follows that the rank of $d\za_1 +td\zb_1$
equals that of $d\tilde\za+td\tilde\zb$ plus two, that is $n^2 +n=dim{\mathcal A}(V,a)-1$.

Finally observe that  ${\mathcal A}(V,a)$  is not unimodular if $a\znoi n$. Thus one may 
choose $a$ in such a way that $\zb_1$ is a contact form, $(d\za_1 ,d\zb_1 )$ is generic 
and ${\mathcal A}(V,a)$ is not unimodular; in this case proposition \ref{pro-3} applied to
${\mathcal A}(V,a)$ and $(\za_1 ,\zb_1 )$ shows the existence on ${\mathcal A}(V,a)^*$
of linear Poisson pairs which are generic and non-flat.}
\end{example}

\section{Lie Poisson pairs and Nijenhuis torsion}\label{sec-8}
In this section a method for constructing Lie Poisson pairs from an endomorphism with
vanishing Nijenhuis is given. 
Recall that the Nijenhuis torsion of a $(1,1)$-tensor field $J$ on a differentiable manifold
is the $(1,2)$-tensor field $N_J$ defined by
$N_J (X,Y)=[JX,JY]+J^2 [X,Y]-J[X,JY]-J[JX,Y]$. 

Let ${\mathcal A}$ be a Lie algebra and
$\zf$ an endomorphism of ${\mathcal A}$ as vector space; since $\zf$ can be seen like a
left invariant $(1,1)$-tensor field on some Lie group, we may define its Nijenhuis torsion,
which in linear terms is given by the formula 
$N_\zf (a,b)=[\zf a,\zf b]+\zf ^2 [a,b]-\zf [a,\zf b]-\zf [\zf a,b]$.

Set $[a,b]_1 = [a,\zf b]+[\zf a,b]-\zf[a,b]$, $a,b\zpe{\mathcal A}$; then
$[a,b]+t[a,b]_1 = [a,(I+t\zf )b]+[(I+t\zf )a,b]-(I+t\zf )[a,b]$. Now assume $N_\zf =0$;
if $I+t\zf$ is invertible then $(I+t\zf )^{-1}[(I+t\zf )a,(I+t\zf )b]
= [a,(I+t\zf )b]+[(I+t\zf )a,b]-(I+t\zf )[a,b]$ since $N_{(I+t\zf )}=0$. Therefore
$[\quad,\quad]+t[\quad,\quad]_1$ defines a structure of Lie algebra and
$(I+t\zf)\zdp({\mathcal A},[\quad,\quad]+t[\quad,\quad]_1 )\zfl
({\mathcal A},[\quad,\quad])$ is an isomorphism of Lie algebras. As $(I+t\zf)$ is 
invertible for almost every $t\zpe\mathbb K$, it follows that $N_\zf =0$ implies 
that $[\quad,\quad]_1$ is a Lie bracket which is compatible with $[\quad,\quad]$.
In this case the couple of Lie-Poisson structures  $(\zL,\zL_1 )$, associated to 
$[\quad,\quad]$ and $[\quad,\quad]_1$ respectively, is a Lie Poisson pair. Moreover 
$(I+t\zf)^* \zdp({\mathcal A}^* ,\zL)\zfl({\mathcal A}^* ,\zL+t\zL_1 )$ is
a Poisson diffeomorphism whenever $(I+t\zf)$ is invertible.    

\begin{example}\label{eje-8}
{\rm Consider the truncated Lie algebra $ {\mathcal A}$ of dimension $m=2n-1\zmai 5$ 
and the basis $\{e_1 ,\zps,e_m \}$ given in example \ref{eje-1}. Let 
$\zf=e_n \zte e_{m}^*$; then $N_\zf =0$. Besides the associated  Lie algebra 
$({\mathcal A},[\quad,\quad]_1 )$ is unimodular and $e_n$ is a basis of its center.

In coordinates $(x_1 ,\zps,x_m )$ with respect to the dual basis 
$\{e_{1}^* ,\zps,e_{m}^* \}$ one has:
$$\zL_1 = x_n \zsu_{i=2}^{n-1}2(i-n){\frac {\zpar} {\zpar x_{i}}}
\zex{\frac {\zpar} {\zpar x_{2n-i}}}\hskip 4 truecm$$
$$\hskip 1truecm+\zpizq(1-n)x_n  {\frac {\zpar} {\zpar x_{1}}}+ \zsu_{i=2}^{n-1}(n-i)
x_{n+i-1}{\frac {\zpar} {\zpar x_{i}}} \zpder\zex {\frac {\zpar} {\zpar x_{2n-1}}} $$
while
$$\zL_{1}^{n-1}=Cx_{n}^{n-1}{\frac {\zpar} {\zpar x_{1}}}\zex\zpu\zpu\zpu\zex
{\frac {\zpar} {\zpar x_{n-1}}}\zex{\frac {\zpar} {\zpar x_{n+1}}}
\zex\zpu\zpu\zpu\zex {\frac {\zpar} {\zpar x_{2n-1}}}$$
where $C$ is a non-vanishing constant. Thus $rank\zL_1 =2n-2$ if $x_n \znoi 0$.
Observe that $dx_n$ is a Casimir of $\zL_1$.

On the other hand 
$$\zL= \zsu_{1\zmei i<j\zmei 2n-1;\, i+j\zmei 2n}(j-i)x_{i+j-1}{\frac {\zpar} {\zpar x_{i}}}
\zex{\frac {\zpar} {\zpar x_{j}}}\hskip 2truecm$$
$$= {\frac {\zpar} {\zpar x_{1}}}\zex\zpizq x_{2}{\frac {\zpar} {\zpar x_{2}}}+\zps
+(2n-2)x_{2n-1}{\frac {\zpar} {\zpar x_{2n-1}}}\zpder\hskip 2truecm$$
$$\hskip 1truecm+ {\frac {\zpar} {\zpar x_{2}}}\zex\zpizq x_{4}{\frac {\zpar} 
{\zpar x_{3}}}+\zps+(2n-4)x_{2n-1}{\frac {\zpar} {\zpar x_{2n-2}}}\zpder$$
$$\hskip 2truecm+\zps+ {\frac {\zpar} {\zpar x_{n-1}}}\zex\zpizq x_{2n-2}
{\frac {\zpar} {\zpar x_{n}}}+2x_{2n-1}{\frac {\zpar} {\zpar x_{n+1}}}\zpder$$

Therefore $\zL^{n-1}(x)\znoi 0$ if and only if the vector fields given by the parentheses
are linearly independent modulo $(\zpar/\zpar x_1 ),\zps,(\zpar/\zpar x_{n-1})$, which
just happens when, at least, $x_{2n-2}\znoi 0$ or $x_{2n-1}\znoi 0$.

Let $A=\{x\zpe{\mathcal A}^* \zbv x_n \znoi 0,x_{2n-2}\znoi 0\}$. Since $\zf$ is 
nilpotent $(I+t\zf)$ is always invertible; moreover 
$(I+t\zf)^* (x)=(x_1 ,\zps,x_{2n-2},x_{2n-1}+tx_n )$, so $(I+t\zf)^* (A)=A$, which
implies that $(\zL+t\zL_1 )^{n-1}(x)\znoi 0$, $t\zpe\mathbb K$, $x\zpe A$, since
$(I+t\zf)^*$ transforms $\zL$ in $\zL+t\zL_1$. Therefore
$(\zL,\zL_1 )$ is generic on $A$.
Finally by lemmas \ref{lem-3} and \ref{lem-5} and remark \ref{rem-5} (first paragraph)
the Lie Poisson  pair $(\zL,\zL_1 )$ is not flat at any point of $A$, because $dx_n$ is
a Casimir of $\zL_1$ and  $rank([e_n ,\quad])\zmai 2$.}
\end{example}

\begin{proposition}\label{pro-7}
Let ${\mathcal A}$ be a non-unimodular Lie algebra of dimension $m=2n-1\zmai 3$. If
there exist $\za,\zb\zpe{\mathcal A}^*$ such that $\zb$ is a contact form and 
$(d\za,d\zb)$ is generic, then on the dual space of the product Lie algebra 
${\mathcal A}\zpor {\mathcal Aff}({\mathbb K})$ there exists some generic and
non-flat Lie Poisson pair.
\end{proposition}

\begin{proof}
Consider a basis $\{\tilde e_1 ,\zps,\tilde e_m \}$ of ${\mathcal A}$ and an another one
$\{\tilde f_1 ,\tilde f_2 \}$ of  ${\mathcal Aff}({\mathbb K})$ such that 
$\zb={\tilde e}_{m}^*$, $d\zb=\zsu_{i=1}^{n-1}{\tilde e}_{2i-1}^*
\zex {\tilde e}_{2i}^*$ and $[\tilde f_1 ,\tilde f_2 ]=\tilde f_1$. Set $e_i =(\tilde e_i ,0)$,
$i=1,\zps,m$, $f_j =(0,\tilde f_j )$, $=1,2$, $\za_1 =\za+f_{1}^*$ and $\zb_1 =\zb$
where $\za,\zb$ are regarded this time like $1$-forms on 
${\mathcal A}\zpor {\mathcal Aff}({\mathbb K})$ in the obvious way (that is 
$\za(\{0\}\zpor {\mathcal Aff}({\mathbb K}))
=\zb(\{0\}\zpor {\mathcal Aff}({\mathbb K}))=0$.) Then $\{e_1 ,\zps,e_m ,f_1 ,f_2 \}$
is a basis of ${\mathcal A}\zpor {\mathcal Aff}({\mathbb K})$; let 
$\{e_{1}^* ,\zps,e_{m}^* ,f_{1}^* ,f_{2}^* \}$ be its dual basis and 
$(x_1 ,\zps,x_m ,y_1 ,y_2 )$ the associated coordinates on
$({\mathcal A}\zpor {\mathcal Aff}({\mathbb K}))^*$

It is easily checked that the Nijenhuis torsion of $f_1\zte\zb_1$ vanishes, so we have a 
second Lie bracket $[\quad,\quad]_1$, which is compatible with the product bracket 
$[\quad,\quad]$. Let $d,d_1$ the respective exterior derivatives and $\zL,\zL_1$ the 
Lie-Poisson structures on $({\mathcal A}\zpor {\mathcal Aff}({\mathbb K}))^*$. We shall 
show that the Lie Poisson pair $(\zL,\zL_1 )$ is generic and non-flat.

First one checks the genericity of  $(\zL,\zL_1 )(\za_1 )$, which is equivalent to that of 
$(d\za_1 ,d_1 \za_1)$. As $d\za_1 =d\za -f_{1}^* \zex f_{2}^*$ and
$d_1 \za_1 =-d\zb-e_{m}^* \zex f_{2}^*$, where $d\za$ and $d\zb$ are computed on 
${\mathcal A}$ and then extended to ${\mathcal A}\zpor {\mathcal Aff}({\mathbb K})$ 
in the natural way, the rank of both $2$-forms equals $2n$. Therefore it will suffices to show
that $rank(d\za_1 +td_1 \za_1)=2n$, $t\zpe{\mathbb C} -\{0\}$, 
which is obvious since $d\za_1 +td_1 \za_1 
=(d\za -td\zb)-(f_{1}^* +te_{m}^* )\zex f_{2}^*$ and
$(d\za,d\zb)$ is generic on ${\mathcal A}$.

Let $(\zw,\zw_1 ,\zW)$ be a representative of $(\zL,\zL_1 )$. As
$({\mathcal A}\zpor {\mathcal Aff}({\mathbb K}),[\quad,\quad]_1 )$ is unimodular 
$d\zw_1 =0$; on the other hand $dy_1$ is a Casimir of $\zL_1$ since $f_1$ belongs to
the center of this algebra. By lemma \ref{lem-5} the vector field 
$X=\zsu_{j=1}^{m}(tr[e_j ,\quad])(\zpar/\zpar x_j )
+\zsu_{k=1}^{2}(tr[f_k ,\quad])(\zpar/\zpar y_k )$ is a basis of $Kerd\zw$; moreover 
since ${\mathcal A}$ is non-unimodular some $(tr[e_j ,\quad])(\zpar/\zpar x_j )$ does
not vanish, so $X\zex(\zpar/\zpar y_1 )\znoi 0$. Assume that $(\zL,\zL_1 )$ is flat
at $\za_1$; then, by lemma \ref{lem-3}, $X$ and 
$\zL(dy_1 ,\quad)=y_1 (\zpar/\zpar y_1 )$ have to be proportional, so
$X\zex(\zpar/\zpar y_1 )=0$ {\it contradiction}. 
\end{proof}

\begin{remark}\label{rem-6}
{\rm The hypotheses of propositions \ref{pro-4} and \ref{pro-7} are rather close and 
almost the same examples illustrate both results. Thus proposition \ref{pro-7} may
be applied to example \ref{eje-3} when $[v,\quad]\zdp {\mathcal A}_0 
\zfl {\mathcal A}_0$ has non-vanishing trace, to example  \ref{eje-4} if 
$a\znoi -1,0,(n-1)^{-1},(n-2)^{-1}$, always to example \ref{eje-5} and, finally, to example
\ref{eje-7} when $a$ is chosen in such a way that ${\mathcal A}(V,a)$ is
non-unimodular, $\zb_1$ is a contact form and $(d\za_1 ,d\zb_1 )$ generic.

Of course example \ref{eje-6} has to be excluded since its Lie algebra is unimodular.}
\end{remark}

\section{Generic non-flat linear Poisson pairs in dimension 3}\label{sec-9}
The purpose of this section and the next one is to illustrate proposition \ref{pro-2}. Of 
course even if the our approach is new, as dimension three
has been well studied, most of the results of
both sections are known, perhaps stated in a different way. 
Let  ${\mathcal A}$ be a $3$-dimensional non-unimodular Lie algebra and $I_0$ its 
unimodular ideal, whose dimension equals two (recall that unimodular implies flatness.) As
Lie algebra $I_0$ itself  is unimodular so abelian; therefore if $u,v\znope I_0$ then 
$[v,\quad]=s[u,\quad]$ for some $s\zpe{\mathbb K}-\{0\}$. Thus, up to non-vanishing
multiplicative constant, one obtain a well-determined endomorphism 
$[u,\quad]_{\zbv I_0}\znoi 0$.

\begin{proposition}\label{pro-8}
Consider a $3$-dimensional non-unimodular Lie algebra ${\mathcal A}$. Then on
${\mathcal A}^*$ there exists some generic linear Poisson  pair that is non-flat if and 
only if the endomorphism $[u,\quad]_{\zbv I_0}$, $u\znope I_0$, is not a multiple
of identity.   
\end{proposition}

Let us see that. Consider a basis $\{e_1 ,e_2 ,e_3 \}$ of ${\mathcal A}$ such that
 $\{e_2 ,e_3 \}$ is a basis of $I_0$. Set $[e_1 ,e_2 ]=a_{22}e_2 +a_{23}e_3$ and
$[e_1 ,e_3 ]=a_{32}e_2 +a_{33}e_3$; note that $a_{22}+a_{33}\znoi 0$ since
${\mathcal A}$ is non-unimodular. Moreover 
$\{e_{1}^* \zex e_{2}^* ,e_{1}^* \zex e_{3}^* \}$ is a basis of the vector space of
$2$-cocycles. In coordinates $(x_1 ,x_2 ,x_3 )$ associates to the dual basis of
$\{e_1 ,e_2 ,e_3 \}$ one has 

$\zL=(\zpar/\zpar x_1 )\zex[(a_{22}x_2 +a_{23}x_3)(\zpar/\zpar x_2 )
+[(a_{32}x_2 +a_{33}x_3)(\zpar/\zpar x_3 )]$,

\noindent so $\zw= -(a_{32}x_2 +a_{33}x_3)dx_2 +(a_{22}x_2 +a_{23}x_3)dx_3$
and $\zW=dx_1 \zex dx_2 \zex  dx_3$ represent $\zL$.

In turn, any constant and $\zL$-compatible Poisson structure $\zL_1$ writes
$\zL_1 =(\zpar/\zpar x_1 )\zex[b_2 (\zpar/\zpar x_2 )
+b_3 (\zpar/\zpar x_3 )]$, $b_2 ,b_3 \zpe\mathbb K$, and 
$\zw_1 =-b_3 dx_2 +b_2 dx_3$, $\zW$ represent it (we do not specify the dependence
on $(b_2 ,b_3 )$ of $\zL_1$.)

On the other hand $(\zL,\zL_1 )(x)$ is generic if and only if
$(a_{22}b_{3}-a_{32}b_{2})x_2 +(a_{23}b_{3}-a_{33}b_{2})x_3 \znoi 0$; therefore the
set of $(b_2 ,b_3 )\zpe{\mathbb K}^2$ such that $(\zL,\zL_1 )$ has no generic point
is always included in a vector line of ${\mathbb K}^2$ (given 
for example by $a_{22}b_{3}-a_{32}b_{2}=0$ or by $a_{23}b_{3}-a_{33}b_{2}=0$; at 
least one of these equations is not trivial since some $a_{ij}$  does not vanish.)  

As $(\zL,\zL_1 )$ is compatible and $d\zw_1 =0$, the $1$-form $\zl$ given by proposition
\ref{pro-2} is a functional multiple of $\zw_1$. Now a computation shows that 

$\zl=(a_{22}+a_{33})[(a_{32}b_{2}-a_{22}b_{3})x_{2}
+(a_{33}b_{2}-a_{23}b_{3})x_{3}]^{-1}(-b_3 dx_2 +b_2 dx_3)$.

Note that $\zl$ is just defined at any generic point of $(\zL,\zL_1 )$. It is easily seen that 
$d\zl=0$ if and only if 

\hskip3truecm $a_{32}b_{2}^{2}+(a_{33}-a_{22})b_{2}b_{3}-a_{23}b_{3}^{2}=0$.

When $[u,\quad]_{\zbv I_0}$, $u\znope I_0$ is a multiple of identity, automatically 
$d\zl=0$ and $(\zL,\zL_1 )$ is flat. Otherwise 
$a_{32}b_{2}^{2}+(a_{33}-a_{22})b_{2}b_{3}-a_{23}b_{3}^{2}$ can be regarded
as a non-trivial quadratic form in $(b_2 ,b_3 )$, which allows to choose 
$(b_2 ,b_3 )\zpe{\mathbb K}^2$ in such a way that the set of generic points of 
$(\zL,\zL_1 )$ is not empty and
$a_{32}b_{2}^{2}+(a_{33}-a_{22})b_{2}b_{3}-a_{23}b_{3}^{2}\znoi0$, so
$d\zl\znoi 0$ and $(\zL,\zL_1 )$ is  non-flat.

\section{Generic non-flat Lie Poisson pairs in dimension 3}\label{sec-10}

This time consider a $3$-dimensional real or complex vector space ${\mathcal A}$ endowed
with two compatible Lie brackets  $[\quad,\quad]$,  $[\quad,\quad]_1$, and their 
respective Lie-Poisson structures $\zL,\zL_1$ on  ${\mathcal A}^*$. One wants to
describe when $(\zL,\zL_1 )$ is generic and non-flat, therefore all flat cases will be put aside.
Observe that at least one of the bracket, for example $[\quad,\quad]$, has to be
 non-unimodular, otherwise  $(\zL,\zL_1 )$ is flat. Now replacing $[\quad,\quad]_1$ by
$[\quad,\quad]_1 +s[\quad,\quad]$ for a suitable scalar $s$ if necessary, allows to suppose
non-unimodular $[\quad,\quad]_1$ as well. Let $I$ be the unimodular ideal of 
 $[\quad,\quad]$ and $I_1$ that of  $[\quad,\quad]_1$; then $(I ,[\quad,\quad])$ and
$(I_1 ,[\quad,\quad]_1 )$ are abelian and $2$-dimensional (see the foregoing section.)
Moreover $I =I_1$. 

Indeed, if $I \znoi I_1$ then ${\mathcal A}=I +I_1$ and there exists a basis
$\{e_1 ,e_2 ,e_3 \}$ of  ${\mathcal A}$ such that $\{e_1 ,e_2 \}$ is a basis of $I$
and $\{e_2 ,e_3 \}$ of $I_1$. Let $(x_1 ,x_2 ,x_3 )$ be the coordinates of
${\mathcal A}^*$ relative to the dual basis. With respect to 
$\zW=dx_1 \zex dx_2 \zex dx_3$ the Lie-Poisson structures $\zL$ and $\zL_1$ are
represented by $\zw=(a_{11}x_{1}+a_{12}x_{2})dx_{1}
+(a_{21}x_{1}+a_{22}x_{2})dx_{2}$ and 
$\zw_1 =(b_{22}x_{2}+b_{23}x_{3})dx_{2}+(b_{32}x_{2}+b_{33}x_{3})dx_{3}$
respectively, where $a_{ij},b_{kr}$ are scalars (see section \ref{sec-9} again.) 

As $d\zw=adx_1 \zex dx_2 \znoi 0$ and $d\zw_1 =bdx_2 \zex dx_3 \znoi 0$ because
$[\quad,\quad]$ and $[\quad,\quad]_1$ are non-unimodular, the $1$-form $\zl$ given by
proposition \ref{pro-2} necessarily equals $fdx_2$ for some function $f$. But
$fdx_2 \zex\zw=d\zw=adx_1 \zex dx_2$ is closed, so $f=f(x_1 ,x_2 )$. On the other
hand, reasoning with $\zw_1$ as before shows that  $f=f(x_2 ,x_3 )$. Therefore
 $f=f(x_2 )$; but in this case $\zl=f(x_2 )dx_2$ is closed and $(\zL,\zL_1 )$ flat. 

In other words, there exists a $2$-dimensional vector subspace $I$ of ${\mathcal A}$
which the unimodular ideal of both brackets. Therefore given any $u\znope I$ the structure
is determined by the restriction of $[u,\quad]$ and $[u,\quad]_1$ to $I$. These two 
endomorphisms of $I$ have non-vanishing trace (obviously the trace before and after
restriction to $I$ is the same), which shows the existence of an unimodular bracket
$s[\quad,\quad]+s_1 [\quad,\quad]_1$, for some $s,s_1 \zpe{\mathbb K}-\{0\}$. Note
that up to non-zero multiplicative constant this bracket is unique. Now replacing 
$[\quad,\quad]_1$ by $s[\quad,\quad]+s_1 [\quad,\quad]_1$ and calling it 
$[\quad,\quad]_1$ again, allows to suppose unimodular $[\quad,\quad]_1$ (of course
one may consider $t[\quad,\quad]+t_1 [\quad,\quad]_1$,
$t\zpe{\mathbb K}-\{0\}$ and $t_1 \zpe{\mathbb K}$, 
instead   $[\quad,\quad]$ if desired.) 

{\it In short, up to linear combinations of brackets, our problem is reduced to consider two 
Lie brackets $[\quad,\quad]$ and $[\quad,\quad]_1$, non-unimodular the first one and
unimodular but non-zero the second one, such that the unimodular ideal $I$ of 
$[\quad,\quad]$ is, at  the same time, an abelian ideal of $[\quad,\quad]_1$. Note that
$[\quad,\quad]$, $[\quad,\quad]_1$ automatically are compatible.}

Observe that the endomorphism $[u,\quad]_{1\zbv I}$, $u\znope I$, is unique up
to non-zero multiplicative constant, while $[u,\quad]_{\zbv I}$ is determined up to
non-vanishing multiplicative constant plus any multiple of $[u,\quad]_{1\zbv I}$. Thus
the existence of some eigenvector of $[u,\quad]_{1\zbv I}$ (perhaps with complex
eigenvalue in the real case) which is not eigenvector of $[u,\quad]_{\zbv I}$ is independent
of the choice of $u$, $[\quad,\quad]$ and $[\quad,\quad]_1$; in other words it is an
intrinsic property of the Lie Poisson pair (more exactly of the Lie Poisson pencil.)

\begin{proposition}\label{pro-9}
The foregoing Lie Poisson pair $(\zL,\zL_1 )$, on ${\mathcal A}^*$, is generic and non-flat
if and only if there exists some eigenvector of $[u,\quad]_{1\zbv I}$ that is not an 
eigenvector of $[u,\quad]_{\zbv I}$.
\end{proposition}

\begin{proof}
Consider a basis $\{e_1 ,e_2 ,e_3 \}$ of ${\mathcal A}$ such that $\{e_2 ,e_3 \}$ is 
a basis of $I$. As $tr([e_1 ,\quad]_{1\zbv I})=0$ but $[e_1 ,\quad]_{1\zbv I}\znoi 0$,
the vector space $I$ is cyclic and one can choose $\{e_2 ,e_3 \}$ in such a way that
$[e_1 ,e_2 ]_1 =e_3$ and $[e_1 ,e_3 ]_1 =be_2$. Set
$[e_1 ,e_2 ]=a_{22}e_{2}+a_{23}e_{3}$, $[e_1 ,e_3 ]=a_{32}e_{2}+a_{33}e_{3}$.
In coordinates $(x_1 ,x_2 ,x_ 3)$ on ${\mathcal A}^*$ associated to the dual basis, and
with respect to $\zW=dx_1 \zex dx_2 \zex dx_3$, the Poisson structures $\zL$ and 
$\zL_1$ are respectively represented by 
$\zw=-(a_{32}x_{2}+a_{33}x_{3})dx_2 +(a_{22}x_{2}+a_{23}x_{3})dx_3$ and
$\zw_1 =-bx_2dx_2 +x_3 dx_3$.

Note that $(\zL,\zL_1 )(x)$ is generic just when $(\zw\zex\zw_1 )(x)\znoi 0$, that is to
say just when $P(x)\znoi 0$ where $P=a_{22}bx_{2}^2 +(a_{23}b-a_{32})x_2 x_3
-a_{33}x_{3}^2$. Therefore $(\zL,\zL_1 )$ is generic if and only if $P$ is not 
identically zero.

As $d\zw_1 =0$ the $1$-form $\zl$ of proposition \ref{pro-2} equals $f\zw_1$ for some
function $f$. On the other hand $d\zw=(a_{22}+a_{33})dx_{2}\zex dx_3 
=\zl\zex\zw=f\zw_1 \zex\zw$ and a computation shows that 
$f=-(a_{22}+a_{33})P^{-1}$. Thus $(\zL,\zL_1 )$ will be flat if and only if
$\zl=-(a_{22}+a_{33})P^{-1}\zw_1$ is closed. Since $\zw_1 =-(1/2)dQ$ where
$Q=bx_{2}^2 -x_{3}^2$, this is equivalent to say that $dP\zex dQ=0$; that is to say, if 
and only if $P$ and $Q$ are proportional as polynomials. Since $Q$ does not identically 
vanish, the foregoing condition is equivalent to the existence of $c\zpe\mathbb K$ such 
that $(a_{22}b,a_{23}b-a_{32},-a_{33})=c(b,0,-1)$.  

First suppose $b=0$. Then $c$ does not exist just when $a_{32}\znoi 0$, that is just when
$e_3$, which determines the single eigendirection of  $[e_1 ,\quad]_{1\zbv I}$, is not an
eigenvector of  $[e_1 ,\quad]_{\zbv I}$. Observe that $a_{32}\znoi 0$ implies 
$P\znoi 0$. 

Now assume $b\znoi 0$. Then $P\znoi 0$ since $a_{22}+a_{33}\znoi 0$. By
replacing  $[\quad,\quad]$ by  $[\quad,\quad]-a_{23}[\quad,\quad]_1$ if necessary,
one may suppose $a_{23}=0$ without lost of generality; in this case $e_2$ is an
eigenvector of $[e_1 ,\quad]_{\zbv I}$. On the other hand 
$(a_{22}b,-a_{32},-a_{33})=c(b,0,-1)$ if and only if $a_{32}=0$ and
$a_{22}=a_{33}$; that is if and only if $[e_1 ,\quad]_{\zbv I}$ is multiple of identity.
Consequently any eigenvector of $[e_1 ,\quad]_{1\zbv I}$ is eigenvector
of $[e_1 ,\quad]_{\zbv I}$ when $c$ exists. 

Conversely, if any eigenvector of $[e_1 ,\quad]_{1\zbv I}$ is eigenvector of 
$[e_1 ,\quad]_{\zbv I}$ then, as $b\znoi 0$, there exist two eigendirections of 
$[e_1 ,\quad]_{\zbv I}$, coming from $[e_1 ,\quad]_{1\zbv I}$, 
which are different from the direction associated to $e_2$; therefore 
$[e_1 ,\quad]_{\zbv I}$ has three distinct eigendirections and necessarily is multiple 
of identity. 
\end{proof}


\end{document}